\documentclass[a4paper,reqno,11pt]{amsart}

\usepackage{amsmath, amsfonts, amssymb, amsthm, amscd,algorithm,algorithmic,enumitem}
\usepackage{graphicx}
\usepackage{psfrag}
\usepackage{perpage}
\usepackage{url}
\usepackage{color}
\usepackage{subfig}
\usepackage{booktabs}

\usepackage[a4paper,scale={0.72,0.74},marginratio={1:1},footskip=7mm,headsep=10mm]{geometry}

\usepackage{hyperref}

\numberwithin{equation}{section}

\newtheorem{theorem}{Theorem}[section]
\newtheorem{lemma}[theorem]{Lemma}
\newtheorem{proposition}[theorem]{Proposition}
\newtheorem{corollary}[theorem]{Corollary}
\newtheorem{remark}[theorem]{Remark}

\newtheorem{example}[theorem]{Example}

\newcommand{\cP}{{\ensuremath{\mathcal P}} }

\renewcommand{\tilde}{\widetilde}          
\DeclareMathSymbol{\leqslant}{\mathalpha}{AMSa}{"36} 
\DeclareMathSymbol{\geqslant}{\mathalpha}{AMSa}{"3E} 
\DeclareMathSymbol{\eset}{\mathalpha}{AMSb}{"3F}     

\newcommand{\R}{\mathbb{R}}

\newcommand{\N}{\mathbb{N}}

\renewcommand{\P}{\ensuremath{\mathbb P}}
\newcommand{\E}{\ensuremath{\mathbb E}}

\renewcommand{\epsilon}{\varepsilon} 
\renewcommand{\rho}{\varrho} 

\newenvironment{myenumerate}{%
  \renewcommand{\theenumi}{\arabic{enumi}}%
\renewcommand{\labelenumi}{{\rm(\theenumi)}}%
\begin{list}{\labelenumi}
	{%
	\setlength{\itemsep}{0.4em}%
	\setlength{\topsep}{0.5em}%
	\setlength\leftmargin{2.45em}%
	\setlength\labelwidth{2.05em}%
	\setlength{\labelsep}{0.4em}%
	\usecounter{enumi}%
	}%
	}%
{\end{list}
}

{\end{list}
}

{\end{myenumerate}}

\newenvironment{myitemize}{%
\begin{list}{$\bullet$}%
 	{%
	\setlength{\itemsep}{0.4em}%
	\setlength{\topsep}{0.5em}%
	\setlength\leftmargin{2.45em}%
	\setlength\labelwidth{2.05em}%
	\setlength{\labelsep}{0.4em}%
	}%
	}%
{\end{list}}

{\end{myitemize}}

\MakePerPage[2]{footnote} 

\newcommand{\tnu}{\tilde{\nu}}
\newcommand{\tmu}{\tilde{\mu}}

\newcommand{\ot}{\overline{t}}
\newcommand{\ut}{\underline{t}}
\newcommand{\oq}{\overline{q}}
\newcommand{\uq}{\underline{q}}

\newcommand{\lecx}{\le_{\textup{cx}}}

\title[Sampling of probability measures in the convex order]{Sampling of probability measures in the convex order by Wasserstein projection}
\author{Aur\'elien Alfonsi, Jacopo Corbetta and Benjamin Jourdain}
\thanks{Universit\'e Paris-Est, Cermics (ENPC), INRIA, F-77455 Marne-la-Vall\'ee, France,\\
e-mails : aurelien.alfonsi@enpc.fr, j.corbetta@zeliade.com, benjamin.jourdain@enpc.fr.\\ This research
benefited from the support of the ``Chaire Risques Financiers'', Fondation du Risque, the French National Research Agency under the program ANR-12-BS01-0019 (STAB) and was completed after the hiring of Jacopo Corbetta by Zeliade Systems}
\date{\today}

\begin{document}
\maketitle

\begin{abstract}
  In this paper, for $\mu$ and $\nu$ two probability measures on $\R^d$ with finite moments of order $\rho\ge 1$, we define the respective projections for the $W_\rho$-Wasserstein distance of $\mu$ and $\nu$ on the sets of probability measures dominated by $\nu$ and of probability measures larger than $\mu$ in the convex order. The $W_2$-projection of $\mu$ can be easily computed when $\mu$ and $\nu$ have finite support by solving a quadratic optimization problem with linear constraints. In dimension $d=1$, Gozlan et al.~\cite{GRSST} have shown that the projections do not depend on $\rho$. We explicit their quantile functions in terms of those of $\mu$ and $\nu$. The motivation is the design of sampling techniques preserving the convex order in order to approximate Martingale Optimal Transport problems by using linear programming solvers. We prove convergence of the Wasserstein projection based sampling methods as the sample sizes tend to infinity and illustrate them by numerical experiments.
\end{abstract}

      
\section{Introduction}

For $\mu,\nu$ in the set ${\mathcal P}(\R^d)$ of probability measures on $\R^d$, we say that $\mu$ is smaller than $\nu$ for the convex order and denote $\mu\lecx\nu$ if $\int_{\R^d}\phi(x)\mu(dx)\le\int_{\R^d}\phi(y)\nu(dy)$ for each convex function $\phi:\R^d\to\R$ non-negative or integrable with respect to $\mu+\nu$.
Up to our knowledge, few studies consider the problem of preserving the convex order while approximating two such probability measures. We can mention the one-dimensional method based on the quantile functions proposed by David Baker in his PhD thesis~\cite{Baker} (see the beginning of Section~\ref{parad1} for more details). 
The dual quantization introduced by Pag\`es and Wilbertz~\cite{PaWi} gives another way to preserve the convex order in dimension one (see the remark after Proposition~10 in~\cite{PaWi}). This is unfortunately no longer true for higher dimensions. Take for example the case of the probability laws $\mu=\delta_{(0,0)}$ and $\nu$ the distribution on $(U,0)$ with $U$ uniform on~$[-1,1]$. We have $\mu \lecx \nu$. We calculate their dual quantizers $\bar{\mu}$ and $\bar\nu$ on the two triangles $\mathcal{T}_1$ and $\mathcal{T}_2$ with vertices $\{(-1,0),(0,-1),(0,1)\}$ and $\{(0,-1),(1,0),(0,1)\}$. We easily obtain $\bar{\mu}=\frac{1}{2}(\delta_{(0,-1)}+\delta_{(0,1)})$, $\bar{\nu}=\frac{1}{4}(\delta_{(0,-1)}+\delta_{(0,1)}+\delta_{(-1,0)}+\delta_{(1,0)} )$. Thus, we have $ \int x^2 \bar{\mu}(dx,dy)=1$,  $\int x^2 \bar{\nu}(dx,dy)=\frac{1}{2}$, which proves that the convex order is not preserved. However, the quantization and the dual quantization give a possible way to approximate $\mu$ and $\nu$ in the convex order. Precisely, the quantization of $\mu$ gives a probability measure $\underline{\mu}$ with finite support such that $\underline{\mu}\lecx\mu$ while the dual quantization of $\nu$ gives a probability measure $\nu$ with finite support such that $\nu\lecx\bar{\nu}$. We therefore have $\underline{\mu}\lecx\bar{\nu}$. Though being general, this construction has several drawbacks. First, to define the dual quantization, $\nu$ and therefore $\mu$ must have a compact support. This is a very restrictive assumption. Second, the calculation of the quantization of $\mu$ and of the dual quantization of $\nu$ is in general not obvious in dimension $d\ge 2$ and may require an important computation time. This is why one usually pre-calculates the quantization for standard distributions, see~\cite{PaPr} for the Gaussian case. Third, this method only works for two measures and does not generalize to design approximations of $\mu,\nu,\eta\in{\mathcal P}(\R^d)$ preserving the convex order when $\mu\lecx\nu\lecx\eta$. 

          To avoid the curse of dimension, it is natural to look at the Monte-Carlo method and to consider the empirical measures $\mu_I=\frac1I \sum_{i=1}^I \delta_{X_i}$ and $\nu_J=\frac1J \sum_{j=1}^J \delta_{Y_j}$, where $X_1,\dots,X_I$ (resp. $Y_1,\dots,Y_J$) are i.i.d. random variables with distribution~$\mu$ (resp. $\nu$). Clearly, there is no reason to have $\frac1I \sum_{i=1}^I X_i = \frac1J \sum_{j=1}^J Y_j$ (a necessary condition for the convex order from the choices $\phi(x_1,\hdots,x_d)=\pm x_k$ with $k\in\{1,\hdots,d\}$) and even more to have   $\mu_I \lecx \nu_J$. In dimension $d=1$, according to Kertz and R\"osler~\cite{KeRo1,KeRo2}, the set of probability measures with a finite first moment is a complete lattice for the increasing and decreasing convex orders. The present paper stems from our preprint~\cite{ACJ1} (Sections 3 and 4), where, in Section 2 devoted to the one-dimensional case, we also investigate the approximation of $\mu_I$ by $\mu_I\wedge \nu_J$ (resp $\nu_J$ by $\mu_I\vee\nu_J$) defined as the infimum of $\mu_I$ and $\nu_J$ for the decreasing convex order when $\frac1I \sum_{i=1}^I X_i\le \frac1J \sum_{j=1}^J Y_j$ and for the increasing convex order otherwise so that $\mu_I\wedge \nu_J\lecx \nu_J$ (resp. $\mu_I\lecx\mu_I\vee\nu_J$). Unfortunately, this approach does not generalize to dimension $d\ge 2$, where, according to Proposition 4.5~\cite{MuSc}, even the set of probability measures with a constant expectation is no longer a lattice for the convex order.
In the present paper, still looking for modifications of $\mu_I$ smaller than $\nu_J$ in the convex order, we introduce the following minimization problem where $\rho\ge 1$
\begin{equation}
   \begin{cases}\text{minimize } \frac{1}{I}\sum_{i=1}^I \left|X_i-\sum_{j=1}^Jr_{ij} Y_j\right|^\rho\\
 \mbox{under the constraints $\forall i,j, \ r_{ij}\ge 0$, $\forall i, \  \sum_{j=1}^J r_{ij}=1$ and $\forall j, \  \sum_{i=1}^I r_{ij}=\frac{I}{J}$}
   \end{cases}\label{miniquad}.
\end{equation}
For $\rho=2$, this is a quadratic optimization problem with linear constraints which can be solved efficiently numerically (see Section~\ref{secnum}).
In general, this is the minimization of a continuous function on a compact set and there exists a minimizer $r^\star$. We then define
$$\mu_{I,J}^{\rho,\star} = \frac{1}{I} \sum_{i=1}^I \delta_{X^\star_i}, \text{ with } X^\star_i=\sum_{j=1}^Jr^\star_{ij} Y_j. $$
By construction, we have $\mu_{I,J}^{\rho,\star}\lecx\nu_J$. In the next section, we generalize this problem by considering, in place of the point measures $\mu_I$ and $\nu_J$, general elements of ${\mathcal P}_\rho(\R^d)=\{\eta\in{\mathcal P}(\R^d):\int_{\R^d}|x|^\rho\eta(dx)<\infty\}$ with $\rho\ge 1$ denoted (with a slight abuse of notation) by $\mu$ and $\nu$. This leads us to define the projection $\mu^\rho_{\underline{\mathcal P}(\nu)}$ of $\mu$ on the set $\underline{\mathcal P}(\nu)=\{\eta\in{\mathcal P}(\R^d):\eta\lecx\nu\}$ of probability measures dominated by $\nu$ in the convex order for the Wasserstein distance with index $\rho$ :
$$W_\rho(\mu,\eta)=\min_{\pi \in \Pi(\mu,\eta)} \left( \int_{\R^d\times \R^d}|x-y|^\rho \pi(dx,dy) \right)^{1/\rho},$$
where $\Pi(\mu,\nu)$ the set of probability measures $\pi$ on $\R^d \times \R^d$ with marginal laws $\mu$ and $\nu$, i.e. $\pi(A\times\R^d)=\mu(A)$ and $\pi(\R^d\times A)=\nu(A)$ for any Borel set $A\subset \R^d$. We show that this projection is well defined for $\rho>1$ and study some of its properties. Notice that after our preprint~\cite{ACJ1}, Gozlan and Juillet~\cite{GoJu} and Backhoff-Varaguas et al.~\cite{B-VBP} have recently considered the projection for $\rho=2$. In dimension $d=1$, according to Gozlan et al.~\cite{GRSST} Theorem 1.5, the projection does not depend on $\rho$. We explicit its quantile function in terms of the quantile functions of $\mu$ and $\nu$ so that it can be computed by efficient algorithms when $\mu$ and $\nu$ have finite supports. In Section~\ref{convexapproxmultidim}, we prove that, when $\mu\lecx\nu$, then $W_\rho(\mu,(\mu_{I})^\rho_{\underline{\mathcal P}(\nu_J)})\le 2W_\rho(\mu,\mu_I)+W_\rho(\nu,\nu_J)$ and deduce that $(\mu_{I})^\rho_{\underline{\mathcal P}(\nu_J)}$ converges weakly to $\mu$ as $I,J\rightarrow +\infty$. Moreover, we extend the construction to the sampling of several probability measures ranked in the convex order. Section~\ref{wasprojnu} is devoted to the projection $\nu^\rho_{\bar{\mathcal P}(\mu)}$ of $\nu$ on the set $\bar{\mathcal P}(\mu)=\{\eta\in{\mathcal P}(\R^d):\mu\lecx\eta\}$ of probability measures larger than $\mu$ in the convex order for the Wasserstein distance with index $\rho$.
Last, in Section~\ref{secnum}, we illustrate by numerical experiments the Wasserstein projection based sampling methods and their application to approximate Martingale Optimal Transport problems. One important motivation of this paper is indeed to tackle numerically the Martingale Optimal Transport (MOT) problem introduced in~\cite{BeHLPe}, which has received a recent and great attention in finance to get model-free bounds on option prices. A family of probability measures on $\R^d$ $(Q(x,dy))_{x\in\R^d}$ is called a Markov kernel on $\R^d$ if for any Borel set $A\subset \R^d$, $\R^d\ni x\mapsto Q(x,A)$ is measurable. We define $\Pi^M(\mu,\nu)=\{ \pi\in \Pi(\mu,\nu):   \forall x\in\R^d,\;\int_{\R^d}|y|\pi_{Y|X}(x,dy)<\infty\mbox{ and }\int_{\R^d}y\pi_{Y|X}(x,dy)=x\}$ where $\pi_{Y|X}$ denotes a Markov kernel such that $\pi(dx,dy)=\mu(dx)\pi_{Y|X}(x,dy)$, the set of martingale couplings. Theorem 8 in Strassen~\cite{Strassen} ensures that, when $\nu\in{\mathcal P}_1(\R^d)$, $\mu \lecx \nu \iff \Pi^M(\mu,\nu)\not = \emptyset$. For a measurable payoff function $c:\R^d \times \R^d \rightarrow \R$, the MOT problem consists in finding an optimal coupling $\pi^\star\in \Pi^M(\mu,\nu)$ that minimizes (or maximizes)
\begin{equation}\label{MOT_continuous}\int_{\R^d \times \R^d} c(x,y) \pi(dx,dy)
\end{equation}
among all couplings $\pi\in \Pi^M(\mu,\nu)$. In finance, this problem arises naturally if one considers the  prices of $d$ assets $S_{T_1},S_{T_2}$ at dates $T_1<T_2$. We assume zero interest rates and suppose that we can observe the marginal laws~$\mu$ (resp. $\nu$) of $S_{T_1}$ (resp. $S_{T_2}$) from option prices on the market and that we want to price an option that pays $c(S_{T_1},S_{T_2})$ at date $T_2$. Any martingale coupling $\pi \in \Pi^M(\mu,\nu)$ is an arbitrage free pricing model: the supremum and the infimum of $\int_{\R^d \times \R^d} c(x,y) \pi(dx,dy)$ over all these couplings give model free bounds on the option price. From the dual formulation of the problem, Beiglb\"ock, Penkner and Henry-Labord\`ere~\cite{BeHLPe} have proved that the upper (resp. lower) bound is the cheapest (resp. most expensive) initial value among superhedging (resp. subhedging) strategies.
 To compute the model free bounds on the option price, one may consider approximating the probability measures $\mu$ and $\nu$ by probability measures with finite supports (typically the empirical measures of i.i.d. samples) ${\mu}_I=\sum_{i=1}^I p_i \delta_{x_i}$ and ${\nu}_J=\sum_{j=1}^J q_j \delta_{y_j}$, with $I,J\in \N^*$, $x_i,y_j\in \R^d, \ p_i,q_j> 0$ for any $i,j$ and $\sum_{i=1}^I p_i=\sum_{j=1}^Jq_j=1$ and solve the approximate MOT problem: to minimize (or maximize) \begin{equation}\label{MOT_discrete}\sum_{i=1}^I\sum_{j=1}^J r_{ij} c(x_i,y_j)
 \end{equation}
over $(r_{ij})_{1\le i\le I,1\le j\le J}$ under the constraints $$r_{ij}\ge 0,\ \sum_{i=1}^I p_ir_{ij}=q_j, \ \sum_{j=1}^J r_{ij}=1 \text{ and }\sum_{j=1}^J r_{ij}y_j=x_i.$$ This problem falls into the realm of linear programming: powerful algorithms have been developed to solve it numerically. The key issue to run these algorithms is the existence of such matrices $(r_{ij})_{1\le i\le I,1\le j \le J}$, that amounts to the existence of a martingale coupling between $\mu_I$ and $\nu_J$. By Strassen's theorem, this is equivalent to have $\mu_I \lecx \nu_J$, which motivates the interest of preserving the convex order when sampling both the probability measures $\mu$ and $\nu$. It is very natural in the financial application to consider empirical measures with $I=J$: once a stochastic model is calibrated to European option market prices, one basically samples it at different times to price exotic options, which gives the empirical measures at those times.

\section{Wasserstein projection of $\mu$ on the set of probability measures dominated by $\nu$ in the convex order} 
\subsection{Definition, existence and uniqueness}
For a Markov kernel $R(x,dy)$ on $\R^d$, we set
$$m_R(x)=\int_{\R^d}yR(x,dy)\mbox{ for }x\in \R^d\mbox{ s.t. }\int_{\R^d}|y|R(x,dy)<\infty.$$
It is well known (see~\cite{DeMe} pages 78--80 or~\cite{Pollard} page 117) that if $\pi \in \Pi(\mu,\nu)$, there exists a $\mu(dx)$-a.e. unique Markov kernel $R$ such that  $\mu(dx)R(x,dy)=\pi(dx,dy)$. This kernel satisfies obviously $\int_{x\in \R^d} \mu(dx)R(x,dy)=\nu(dy)$, which we note $\mu R = \nu$ later on. Conversely, if $R$ is a kernel satisfying $\mu R = \nu$ then $\mu(dx)R(x,dy)$ defines a probability measure in~$\Pi(\mu,\nu)$.
We define $\cP(\R^d)$ the set of probability measures on $\R^d$ and, for $\rho \ge 1$, 
$$\mathcal{P}_\rho(\R^d)=\{ \mu\in \cP(\R^d), \  \int_{\R^d}|x|^\rho \mu(dx)< \infty \},$$
the set of probability measures with finite moment of order~$\rho$. 

Suppose that $\nu\in\mathcal{P}_1(\R^d)$ and $R$ is a Markov kernel such that $\mu R=\nu$. Then $$\int_{\R^d\times\R^d}|y|R(x,dy)\mu(dx)=\int_{\R^d}|y|\nu(dy)<\infty$$ so that $m_R(x)$ is defined $\mu(dx)$-a.e.. Moreover for each convex function $\phi:\R^d\rightarrow \R$ such that $\sup_{x\in\R^d}\frac{|\phi(x)|}{1+|x|}<\infty$, by Jensen's inequality,  
\begin{align*}
   \int_{\R^d}\phi(y)\nu(dy)&=\int_{\R^d\times\R^d}\phi(y)\mu(dx)R(x,dy)\\&\ge \int_{\R^d}\phi\left(\int_{\R^d}yR(x,dy)\right)\mu(dx)=\int_{\R^d}\phi(m_R(x))\mu(dx).
\end{align*}
Despite the restriction on the growth of the convex function $\phi$, by Lemma~\ref{lem_ordre_cvx} below, this ensures that $m_R\#\mu\lecx\nu$.

For $\rho\ge 1$ and $\mu,\nu\in{\mathcal P}_\rho(\R^d)$,  we consider the following generalization of the minimization problem \eqref{miniquad} :
\begin{equation*}
   \begin{cases}
      \mbox{Minimize }{\mathcal J}_\rho(R):=\int_{\R^d}|x-m_R(x)|^\rho\mu(dx)\\
\mbox{under the constraint that $R$ is a Markov kernel such that $\mu R=\nu$}
   \end{cases}.
\end{equation*}
Note that this problem is a particular case of the general transport costs considered by Gozlan et al.~\cite{GRST} and Alibert et al. ~\cite{ABC}, who are interested in duality results and by Backhoff-Veraguas et al.~\cite{B-VBP} who deal with existence of optimal transport plans and necessary and sufficient optimality conditions in the spirit of cyclical monotonicity. Gozlan and Juillet~\cite{GoJu} characterize optimal transport plans between $\mu$ and $\nu$ for the cost $J_2$.
When the $X_i$ are distinct, \eqref{miniquad} is recovered by setting $$R(x,dy)=\begin{cases}\sum_{j=1}^J r_{ij}\delta_{Y_j}(dy)\mbox{ if }x=X_i\mbox{ for some }i\in\{1,\hdots,I\}\\
   \delta_x(dy)\mbox{ if } x\notin\{X_1,\hdots,X_I\}\end{cases}.$$
At optimality in \eqref{miniquad}, by Jensen's inequality $\sum_{j=1}^J r_{ij}Y_j=\sum_{j=1}^J r_{kj}Y_j$ when $X_i=X_k$ for $1\le k\ne i\le I$ and the problem \eqref{miniquad}  modified with the additional constraint $\sum_{j=1}^Jr_{ij}Y_j=\sum_{j=1}^Jr_{kj}Y_j$ when $X_i=X_k$ is recovered by setting
$$R(x,dy)=\begin{cases}\frac{1}{\sum_{i=1}^I 1_{\{X_i=x\}}}\sum_{i:X_i=x}\sum_{j=1}^J r_{ij}\delta_{Y_j}(dy)\mbox{ if }x\in\{X_1,\hdots,X_I\}\\
   \delta_x(dy)\mbox{ if } x\notin\{X_1,\hdots,X_I\}\end{cases}.$$
According to the next theorem the generalized problem is equivalent to the computation of the projection of $\mu$ on the set of probability measures dominated by $\nu$ in the convex order for the $\rho$-Wasserstein distance.

\begin{theorem}\label{lem_couplage_Rd_gen}
  Let $\rho\ge 1$, $\mu,\nu \in \mathcal{P}_\rho(\R^d)$. One has $\inf_{R:\mu R=\nu}\mathcal{J}_\rho(R)=\inf_{\eta\in\underline{\mathcal P}(\nu)}W_\rho^\rho(\mu,\eta)$ where both infima are attained. If $\rho>1$, then the functions $\{m_{R_\star}:\;\mu R_\star=\nu\mbox{ and }\mathcal{J}_\rho(R_\star)=\inf_{R:\mu R=\nu}\mathcal{J}_\rho(R)\}$ are $\mu(dx)$ a.e. equal, $\mu^\rho_{\underline{\mathcal P}(\nu)}:=m_{R_\star}\#\mu$ is the unique $\eta\lecx\nu$ minimizing  $W_\rho^\rho(\mu,\eta)$ and $\mu(dx)\delta_{m_{R_\star}(x)}(dy)$ the unique optimal transport plan $\pi\in\Pi(\mu,\mu^\rho_{\underline{\mathcal P}(\nu)})$ such that $W_\rho^\rho(\mu,\mu^\rho_{\underline{\mathcal P}(\nu)})=\int_{\R^d\times \R^d}|x-y|^\rho \pi(dx,dy)$.
\end{theorem}
When $\rho>1$, $\mu^\rho_{\underline{\mathcal P}(\nu)}$ is the projection of $\mu$ on the set of probability measures dominated by $\nu$ in the convex order and $\mu^\rho_{\underline{\mathcal P}(\nu)}\lecx\nu$. \begin{proof}
For $\eta\in\cP(\R^d)$, 
\begin{align*}
   W_\rho^\rho(\mu,\eta)\le \int_{\R^d\times\R^d}|x-y|^\rho\mu(dx)\eta(dy)\le 2^{\rho-1}\left(\int_{\R^d}|x|^\rho\mu(dx)+\int_{\R^d}|y|^\rho\eta(dy)\right).
\end{align*}
where the right-hand side is finite if $\eta\in\underline{\mathcal P}(\nu)$ since $\sup_{\eta\in\underline{\mathcal P}(\nu)}\int_{\R^d}|x|^\rho\eta(dx)=\int_{\R^d}|x|^\rho\nu(dx)$. By the Markov inequality and the Prokhorov theorem, this last bound implies that $\underline{\mathcal P}(\nu)$ is relatively compact for the weak convergence topology. For $K\in(0,\infty)$ and $\eta\lecx\nu$, denoting by $R$ a martingale kernel such that $\eta R=\nu$, we have
\begin{align*}
   \int_{\R^d}|x|1_{\{|x|\ge K\}}\eta(dx)&=\int_{\R^d}\left|\int_{\R^d}yR(x,dy)\right|1_{\{|x|\ge K\}}\eta(dx)\le \int_{\R^d\times \R^d}|y|1_{\{|x|\ge K\}}R(x,dy)\eta(dx)\\&\le \int_{\R^d\times \R^d}|y|1_{\{|y|\ge \sqrt{K}\}}R(x,dy)\eta(dx)+\sqrt{K}\int_{\R^d}1_{\{|x|\ge K\}}\eta(dx)\\&\le \int_{\R^d}|y|1_{\{|y|\ge \sqrt{K}\}}\nu(dy)+\frac{\int_{\R^d}|x|\eta(dx)}{\sqrt{K}}\\&\le \int_{\R^d}|y|1_{\{|y|\ge \sqrt{K}\}}\nu(dy)+\frac{\int_{\R^d}|y|\nu(dy)}{\sqrt{K}}.
\end{align*}
For $(\eta_n)_n$ a sequence in $\{\eta\in{\mathcal P}(\R^d):\eta\lecx\nu\}$ weakly converging to $\eta_\infty$, this implies uniform integrability ensuring that for $\phi:\R^d\to\R$ continuous and such that $\sup_{x\in\R^d}\frac{|\phi(x)|}{1+|x|}<\infty$, $\lim_{n\to\infty}\int_{\R^d}\phi(x)\eta_n(dx)=\int_{\R^d}\phi(x)\eta_\infty(dx)$. With Lemma~\ref{lem_ordre_cvx} below and the continuity of real valued convex functions on $\R^d$, we deduce that $\eta_\infty\in\underline{\mathcal P}(\nu)$. Hence $\underline{\mathcal P}(\nu)$ is compact for the weak convergence topology.

Since $\eta\mapsto W_\rho^\rho(\mu,\eta)$ is lower-semicontinuous for this topology, there exists $\eta_\star\in\underline{\mathcal P}(\nu)$ such that $W_\rho^\rho(\mu,\eta_\star)=\inf_{\eta\in\underline{\mathcal P}(\nu)}W_\rho^\rho(\mu,\eta)$.  Let $P$ be a martingale Markov kernel such that $\eta_\star P=\nu$ and $Q$ a Markov kernel such that $\mu Q=\eta_\star$ and
$W_\rho^\rho(\mu,\eta_\star)=\int_{\R^d\times\R^d}|x-y|^\rho Q(x,dy)\mu(dy)$.
One has $\mu QP=\eta_\star P=\nu $ and, by martingality of $P$,  $$m_{QP}(x)=\int_{\R^d\times\R^d}zP(y,dz)Q(x,dy)=\int_{\R^d}y Q(x,dy).$$ With Jensen's inequality, we deduce that
\begin{equation}
   W_\rho^\rho(\mu,\eta_\star)=\int_{\R^d\times\R^d}|x-y|^\rho Q(x,dy)\mu(dy)\ge \int_{\R^d}\left|x-\int_{\R^d}yQ(x,dy)\right|^\rho \mu(dx)=\mathcal{J}_\rho(QP).\label{mrq}
\end{equation}

On the other hand, for any Markov kernel $R$ such that $\mu R=\nu$, $m_R\#\mu\lecx \nu$ and $\mathcal{J}_\rho(R)=\int_{\R^d}|x-m_{R}(x)|^\rho\mu(dx)\ge W_\rho^\rho(\mu,m_{R}\#\mu)$. Hence $$\inf_{R:\mu R=\nu}\mathcal{J}_\rho(R)\ge \inf_{\eta\in\underline{\mathcal P}(\nu)}W_\rho^\rho(\mu,\eta)=W_\rho^\rho(\mu,\eta_\star)\ge \mathcal{J}_\rho(QP)\ge \inf_{R:\mu R=\nu}\mathcal{J}_\rho(R)$$
so that both infima are equal and $\mathcal{J}_\rho(QP)=\inf_{R:\mu R=\nu}\mathcal{J}_\rho(R)$. Moreover, the inequality in \eqref{mrq} is an equality. If $\rho>1$, by strict convexity of $x\mapsto |x|^\rho$, this implies that $\mu(dx)$ a.e. $R(x,dy)=\delta_{m_{QP}(x)}(dy)$ so that $\eta_\star=\mu Q=m_{QP}\#\mu$.

  For $\rho>1$, the uniqueness of $m_{R_\star}$ is also obtained from the strict convexity of $x\mapsto |x|^\rho$. Namely, for any  optimal kernel $R_\star$ we have
  \begin{align*}
    & \mathcal{J}_\rho((R_\star+ QP)/2) =  \int_{\R^d} \left|x-\frac{m_{R_\star}(x)+m_{QP}(x)}{2} \right|^\rho \mu(dx)  \\
    &\le \int_{\R^d} \frac 12 |x-m_{R_\star}(x)|^\rho +\frac 12|x-m_{QP}(x)|^\rho \mu(dx)=\frac{1}{2}(\mathcal{J}_\rho(R_\star)+\mathcal{J}_\rho(QP))=\inf_{R:\mu R=\nu}\mathcal{J}_\rho(R).  
  \end{align*}
  Since $\mu\frac{R_\star+ QP}2=\nu$, we necessarily have $\mathcal{J}_\rho((R_\star+ QP)/2)=\inf_{R:\mu R=\nu}\mathcal{J}_\rho(R)$ and then $m_{R_\star}(x)=m_{QP}(x)$, $\mu(dx)$-a.e..
\end{proof}

\begin{remark} When $\rho=1$, let us give an example of non-uniqueness for the optimal functions $m_R$ and the probability measures $\eta_\star\in\underline{\mathcal P}(\nu)$ such that $W_1(\mu,\eta_\star)=\inf_{\eta\in\underline{\mathcal P}(\nu)}W_1(\mu,\eta)$.
  Let $\mu(dx)=1_{[0,1]}(dx)$ (resp. $\nu(dy)=1_{[1,2]}(dy)$) be the uniform law on $[0,1]$ (resp. $[1,2]$). We have
  $$\inf_{R:\mu R=\nu}\mathcal{J}_1(R)\ge\inf_{R:\mu R=\nu} \left|\int_{\R} x \mu(dx)-\int_{\R} m_R(x) \mu(dx)\right|= \left|\int_{\R} x \mu(dx)-\int_{\R} y \nu(dy)\right|=1.$$
For $\lambda \in [0,1]$, $R_\lambda(x,dy)=(1-\lambda) \delta_{1+x}(dy)+\lambda\delta_{2-x}(dy) $ is such that $\mu R_\lambda=\nu$, $m_{R_\lambda}(x)=(1+\lambda)+ (1-2\lambda)x$ and $m_{R_\lambda}\#\mu$ is the uniform law on $[(1+\lambda)\wedge(2-\lambda),(1+\lambda)\vee(2-\lambda)]$. Using that $m_{R_\lambda}(x)\ge 1$ for $x\in (0,1)$ for the first equality, we have
$$W_1(\mu,m_{R_\lambda}\#\mu)\le\mathcal{J}_1(R_\lambda)=\int_0^1 1+\lambda -2\lambda x dx=1=\inf_{\eta\in\underline{\mathcal P}(\nu)}W_1(\mu,\eta).$$
Thus all the kernels $R_\lambda$ and pushforward measures $m_{R_\lambda}\#\mu$ are optimal.
\end{remark}

\begin{example}\label{exemple_translation} Let $\mu\lecx \nu$ and $\rho\ge 1$. We assume that $\nu\in \cP_\rho$, which implies that $\mu \in \cP_\rho$. For $\alpha \in \R^d$, let $\mu^\alpha$ be the image of $\mu$ by $x\mapsto x+\alpha$. Then, for any kernel $R$ such that $\mu^\alpha R = \nu$,
  \begin{align*}
    \int_{\R^d} |x-m_R(x)|^\rho &\mu^\alpha (dx) \ge \left| \int_{\R^d} x-m_R(x)\mu^\alpha(dx)\right|^\rho \\&=\left| \int_{\R^d} x\mu^\alpha(dx)- \int_{\R^d}y \nu(dy) \right|^\rho=\left| \int_{\R^d} x\mu^\alpha(dx)- \int_{\R^d} x\mu(dx) \right|^\rho=|\alpha|^\rho.
  \end{align*}
This lower bound is attained for $R(x,dy)=Q(x-\alpha,dy)$, where $Q$ is any martingale kernel such that $\mu Q= \nu$, since $m_R(x)=x-\alpha$ for this choice. Therefore, for $\rho>1$, $(\mu^{\alpha})^\rho_{\underline{\mathcal P}(\nu)}=\mu$.
\end{example}
Let us observe that if $\mu,\nu \in \mathcal{P}_\rho(\R^d)$ with $\rho>1$, then we have $\mu,\nu \in \mathcal{P}_{\rho'}(\R^d)$ for any $\rho'\in (1,\rho)$. In general, as in the next example, $\mu^{\rho'}_{\underline{\mathcal P}(\nu)}$ is different from $\mu^\rho_{\underline{\mathcal P}(\nu)}$.
\begin{example}
  Let $d=2$, $\mu=\frac{1}{2}\left(\delta_{(1,0)}+\delta_{(0,a)}\right)$ with $a\in\R$ and $\nu=\frac{1}{2}\left(\delta_{(1,0)}+\delta_{(-1,0)}\right)$. For $\rho>1$, since, by Theorem~\ref{lem_couplage_Rd_gen}, $\mu^\rho_{\underline{\mathcal P}(\nu)}$ is the image of $\mu$ by some transport map and $\mu^\rho_{\underline{\mathcal P}(\nu)}\lecx\nu$, one has $\mu^\rho_{\underline{\mathcal P}(\nu)}=\frac{1}{2}\left(\delta_{(x_\rho,0)}+\delta_{(-x_\rho,0)}\right)$ for some $x_\rho\in[0,1]$. For $x\in[0,1]$, since the distances between $(x,0)$ and $(0,a)$ and between $(-x,0)$ and $(0,a)$ are equal whereas $(x,0)$ is closer to $(1,0)$ than $(-x,0)$,  one has $2W_\rho^\rho(\mu,\frac{1}{2}\left(\delta_{(x,0)}+\delta_{(-x,0)}\right))=(1-x)^\rho+(a^2+x^2)^{\rho/2}$. Since the unique minimizer of $x\mapsto (1-x)^2+(a^2+x^2)$ on $[0,1]$ is $\frac{1}{2}$, one has $x_2=\frac{1}{2}$ and $\mu^2_{\underline{\mathcal P}(\nu)}=\frac{1}{2}\left(\delta_{(1/2,0)}+\delta_{(-1/2,0)}\right)$. Since the unique minimizer of $x\mapsto (1-x)^3+(5+x^2)^{3/2}$ on $[0,1]$ is $\frac{1}{4}$, for $a\in\{-\sqrt{5},\sqrt{5}\}$, one has $x_3=\frac{1}{4}$ and $\mu^3_{\underline{\mathcal P}(\nu)}=\frac{1}{2}\left(\delta_{(1/4,0)}+\delta_{(-1/4,0)}\right)$.
 \end{example}

Nevertheless, the situation is strikingly different in dimension $d=1$ where, according to Gozlan et al.~\cite{GRSST} Theorem 1.5, the projection does not depend on $\rho$. We are going to explicit this projection by characterizing its quantile function in terms of the quantile functions of $\mu$ and $\nu$.

\subsection{Dimension $d=1$}\label{parad1}
Let $F_{\mu}(x)=\mu((-\infty,x])$ and $F_{\nu}(x)=\nu((-\infty,x])$ be the cumulative distribution functions and for $p \in (0,1)$, $F_\mu^{-1}(p)=\inf \{x \in \R : F_\mu(x)\ge p \}$ and $F_\nu^{-1}(p)=\inf \{x \in \R : F_\nu(x)\ge p \}$ their left-continuous and non-decreasing generalized inverses also called quantile functions. The convex order is characterized as follows in terms of the quantile functions (see Theorem 3.A.5~\cite{ShSh}) : for $\mu,\nu\in{\mathcal P}_1(\R)$,
\begin{equation}
   \mu\lecx\nu\mbox{ iff }\int_0^1F_\mu^{-1}(p)dp=\int_0^1F_\nu^{-1}(p)dp\mbox{ and }\forall q\in(0,1),\;\int_q^1F_\mu^{-1}(p)dp\le\int_q^1F_\nu^{-1}(p)dp.\label{caraccxquant}
\end{equation}
Notice that, as a consequence of this characterization, if $\mu\lecx\nu$, then for $I,k\ge 1$, $\frac{1}{I}\sum_{i=1}^I\delta_{I \int_{\frac{i-1}{I}}^{\frac{i}{I} } F_\mu^{-1}(u) du }\lecx\frac{1}{kI}\sum_{j=1}^{kI}\delta_{kI \int_{\frac{j-1}{kI}}^{\frac{j}{kI} } F_\nu^{-1}(u) du }$, as stated by Baker in Theorem~2.4.11~\cite{Baker}. 
\begin{theorem}\label{propprojd1}
   For $\mu,\nu\in{\mathcal P}_1(\R)$, let $\psi$ denote the convex hull (largest convex function bounded from above by) of the function $[0,1]\ni q\mapsto\int_0^qF_{\mu}^{-1}(p)-F_\nu^{-1}(p)dp$. There exists a probability measure $\mu_{\underline{\mathcal P}(\nu)}$ such that $\forall q\in[0,1]$, $\int_0^qF_{\mu_{\underline{\mathcal P}(\nu)}}^{-1}(p)dp=\int_0^qF_{\mu}^{-1}(p)dp-\psi(q)$. Moreover, $\mu_{\underline{\mathcal P}(\nu)}\in\underline{\mathcal P}(\nu)$ and for each $\rho>1$ such that $\mu,\nu\in{\mathcal P}_\rho(\R)$, $\mu^\rho_{\underline{\mathcal P}(\nu)}=\mu_{\underline{\mathcal P}(\nu)}$. Last, $T(x)=F_\mu^{-1}(F_\mu(x))-\psi'(F_\mu(x)-)$ is non-decreasing and is an optimal transport map : $T\#\mu=\mu_{\underline{\mathcal P}(\nu)}$ and for all~$\rho\ge 1$, $W_\rho^\rho(\mu,\mu_{\underline{\mathcal P}(\nu)})= \int_{\R} |T(x)-x|^\rho \mu(dx)$.
\end{theorem}
For probability measures $\mu_I=\sum_{i=1}^Ip_i\delta_{x_i}$ (resp. $\nu_J=\sum_{j=1}^Jq_j\delta_{y_j}$) on the real line with $(p_1,\hdots,p_I)\in(0,1]^I$ and $x_1<x_2<\hdots<x_I$ (resp. $(q_1,\hdots,q_J)\in(0,1]^J$ and $y_1<y_2<\hdots<y_J$), the continuous and piecewise affine function $q\mapsto\int_0^qF_{\mu_I}^{-1}(p)-F_{\nu_J}^{-1}(p)dp$ changes slope at $q\in\{\sum_{k=1}^ip_k:1\le i\le I-1\}\cup \{\sum_{k=1}^jq_k:1\le j\le J-1\}$ with a change equal to $\sum_{i=1}^{I-1}1_{\{q=\sum_{k=1}^ip_k\}}(x_{i+1}-x_{i})-\sum_{j=1}^{J-1}1_{\{q=\sum_{k=1}^j q_k\}}(y_{j+1}-y_j)$ (which can be equal to zero if $q=\sum_{k=1}^ip_k=\sum_{k=1}^j q_k$ and $x_{i+1}-x_{i}=y_{j+1}-y_j$). Clearly, $\psi$ is piecewise affine and changes slope at most at points $q\in\{\sum_{k=1}^ip_k:1\le i\le I-1\}$ with changes not greater than $x_{i+1}-x_i$ so that $(\mu_{I})_{\underline{\mathcal P}(\nu_J)}=\sum_{i=1}^Ip_i\delta_{z_i}$ with $z_1\le z_2\le \hdots\le z_I$. The convex hull $\psi$ can be computed by Andrew's monotone chain algorithm and the points $(z_i)_{i\in I}$ are easily deduced. 

The proof of Theorem~\ref{propprojd1} relies on the following lemma and is postponed after its proof.
\begin{lemma}\label{lem_couplage_R} Let $\rho>1$ and $\mu,\nu\in{\mathcal P}_\rho(\R)$. Then
$(0,1)\ni p\mapsto F_{\mu^\rho_{\underline{\mathcal P}(\nu)}}^{-1}(p)-F_\mu^{-1}(p)$ is non-increasing.
\end{lemma}
\begin{proof}
It is enough to check that if $p\mapsto F_\eta^{-1}(p)-F^{-1}_\mu(p)$ is not non-increasing for some $\eta\in\underline{\mathcal P}(\nu)$, one can find $\tilde\eta\in\underline{\mathcal P}(\nu)$ such that $W_\rho^\rho(\mu,\tilde\eta)<W_\rho^\rho(\mu,\eta)$ where, according to Proposition 2.17~\cite{santambrogio}, $W_\rho^\rho(\mu,\eta)=\int_0^1|F_\eta^{-1}(p)-F_\mu^{-1}(p)|^\rho dp$.
With the left-continuity of $p\mapsto F_\eta^{-1}(p)-F^{-1}_\mu(p)$, the lack of monotonicity of this function is equivalent to
$$0<\int_{(0,1)^2}1_{I_\eta}(p,q)dpdq\mbox{ where }I_\eta=\{(p,q):(p-q)(F_\eta^{-1}(p)-F^{-1}_\mu(p)-F_\eta^{-1}(q)+F^{-1}_\mu(q))>0\}.$$
Let $\alpha(p,q)=1_{I_\eta}(p,q)\frac{F_\eta^{-1}(p)-F^{-1}_\mu(p)-F_\eta^{-1}(q)+F^{-1}_\mu(q)}{2(F_\eta^{-1}(p)-F_\eta^{-1}(q))}$, where one easily checks that the denominator does not vanish on $I_\eta$ and that $0\le\alpha(p,q)=\alpha(q,p)<1$. For $(p,q)\in I_\eta$, $$\alpha(p,q)F_\eta^{-1}(q)+(1-\alpha(q,p))F_\eta^{-1}(p)-F^{-1}_\mu(p)=\frac{F_\eta^{-1}(p)-F^{-1}_\mu(p)+F_\eta^{-1}(q)-F^{-1}_\mu(q)}{2},$$ so that by strict convexity,
\begin{align*}
 \frac{1}{2}\bigg(|F_\eta^{-1}(p)-F^{-1}_\mu(p)|^\rho&+|F_\eta^{-1}(q)-F^{-1}_\mu(q)|^\rho\bigg)\\>  &|\alpha(p,q)F_\eta^{-1}(q)+(1-\alpha(q,p))F_\eta^{-1}(p)-F^{-1}_\mu(p)|^\rho.
\end{align*}
With Jensen's inequality, we deduce that
\begin{align*}
 W_\rho^\rho(\mu,\eta)=&\frac{1}{2}\int_{(0,1)^2}|F_\eta^{-1}(p)-F^{-1}_\mu(p)|^\rho+|F_\eta^{-1}(q)-F^{-1}_\mu(q)|^\rho dpdq\\>&\int_{(0,1)^2}|\alpha(p,q)F_\eta^{-1}(q)+(1-\alpha(q,p))F_\eta^{-1}(p)-F^{-1}_\mu(p)|^\rho dpdq\\\ge& \int_0^1\left|\int_0^1\alpha(p,q)F_\eta^{-1}(q)+(1-\alpha(q,p))F_\eta^{-1}(p)dq-F^{-1}_\mu(p)\right|^\rho dp.
\end{align*}
The right-hand side is not smaller than $W_\rho^\rho(\mu,\tilde \eta)$ where $\tilde \eta$ denotes the image of the Lebesgue measure on $(0,1)$ by $p\mapsto \int_0^1\alpha(p,q)F_\eta^{-1}(q)+(1-\alpha(q,p))F_\eta^{-1}(p)dq$. For $\phi:\R\to\R$ convex and such that $\sup_{x\in\R}\frac{|\phi(x)|}{1+|x|^\rho}<\infty$, by Jensen's inequality, \begin{align*}
\int_\R\phi(x)\tilde\eta(dx)\le
   \int_{(0,1)^2}\alpha(p,q)\phi(F_\eta^{-1}(q))+(1-\alpha(q,p))\phi(F_\eta^{-1}(p))dqdp=\int_0^1\phi(F_\eta^{-1}(q))dq.
\end{align*}
Since the right-hand side is equal to $\int_\R\phi(x)\eta(dx)$, by Lemma~\ref{lem_ordre_cvx} below, one has $\tilde\eta\in\underline{\mathcal P}(\nu)$.

\end{proof}
 
\begin{proof}[Proof of Theorem~\ref{propprojd1}] 
Let $U$ be uniformly distributed on $(0,1)$.

Since for all $q\in[0,1]$, $\int_0^qF_\mu^{-1}(p)-F_\nu^{-1}(p)dp\ge \int_0^qF_\mu^{-1}(p)dp-q\int_0^1F_\nu^{-1}(p)dp$ where the right-hand side is a convex function of $q$, one has $\psi(0)=0$ and $\psi(1)=\int_0^1F_\mu^{-1}(p)-F_\nu^{-1}(p)dp$. By Lemma~\ref{lemconv} below, the convexity of both $q\mapsto \int_0^qF_{\mu}^{-1}(p)dp$ and $q\mapsto \int_0^qF_{\nu}^{-1}(p)dp$ implies that $q\mapsto\int_0^qF_{\mu}^{-1}(p)dp-\psi(q)$ is convex. Let $f$ denote the left-hand derivative of this function and $\mu_{\underline{\mathcal P}(\nu)}$ the probability distribution of $f(U)$. By Lemma~\ref{lemquant} below, $f$ is equal to $F^{-1}_{\mu_{\underline{\mathcal P}(\nu)}}$ so that $\forall q\in[0,1]$, $\int_0^qF_{\mu_{\underline{\mathcal P}(\nu)}}^{-1}(p)dp=\int_0^qF_{\mu}^{-1}(p)dp-\psi(q)$.

Let $q\in[0,1]$. Since $\psi(q)\le \int_0^qF_\mu^{-1}(p)-F_\nu^{-1}(p)dp$ with equality when $q=1$, one has $\int_0^qF_{\mu_{\underline{\mathcal P}(\nu)}}^{-1}(p)dp=\int_0^qF_{\mu}^{-1}(p)dp-\psi(q)\ge \int_0^q F_\nu^{-1}(p)dp$ with equality when $q=1$ so that by~\eqref{caraccxquant}, $\mu_{\underline{\mathcal P}(\nu)}\lecx\nu$.
By concavity of $[0,1]\ni q\mapsto -\psi(q)=\int_0^q F_{\mu_{\underline{\mathcal P}(\nu)}}^{-1}(p)-F_\mu^{-1}(p)dp$, the left-continuous function $(0,1)\ni p\mapsto F_{\mu_{\underline{\mathcal P}(\nu)}}^{-1}(p)-F_\mu^{-1}(p)$ is non-increasing.

The set $$\tilde{\mathcal P}(\nu):=\{\eta\in\underline{\mathcal P}(\nu): (0,1)\ni p\mapsto F_\eta^{-1}(p)-F_\mu^{-1}(p)\mbox{ is non-increasing}\}$$ is not empty since $\mu_{\underline{\mathcal P}(\nu)},\delta_{\int_\R y\nu(dy)}\in\tilde{\mathcal P}(\nu)$. 
Let $D(\eta)$ denote the distribution of $F_\eta^{-1}(1-U)-F_\mu^{-1}(1-U)$ for $\eta\in\tilde{\mathcal P}(\nu)$. For all $\eta\in\tilde{\mathcal P}(\nu)$, $\int_\R|x|D(\eta)(dx)<\infty$ and $\int_\R xD(\eta)(dx)=\E\left[F_\eta^{-1}(1-U)-F_\mu^{-1}(1-U)\right]=\int_\R x(\nu-\mu)(dx)$. By Lemma~\ref{leminf} below,  the set $\{D(\eta):\eta\in\tilde{\mathcal P}(\nu)\}$ admits an infimum $\pi$ for the convex order and for all $q\in[0,1]$, $\int_q^1F_{\pi}^{-1}(p)dp=\inf_{\eta \in\tilde{\mathcal P}(\nu)}\int_q^1F_{D(\eta)}^{-1}(p)dp$. For $\eta\in\tilde{\mathcal P}(\nu)$, since $(0,1)\ni p\mapsto F_\eta^{-1}(1-p)-F^{-1}_\mu(1-p)$ is non-decreasing, by Lemma~\ref{lemquant}, $p\mapsto F_{D(\eta)}^{-1}(p)$ and $p\mapsto F_\eta^{-1}(1-p)-F^{-1}_\mu(1-p)$ coincide away from the at most countable set of their common discontinuities, with the former left-continuous and the latter right-continuous. Hence for $q\in[0,1]$, $$\int_q^1F_{\pi}^{-1}(p)dp=\inf_{\eta\in\tilde{\mathcal P}(\nu)}\int_0^{1-q}F_\eta^{-1}(p)-F^{-1}_\mu(p)dp=-\sup_{\eta\in\tilde{\mathcal P}(\nu)}\int_0^{1-q}F^{-1}_\mu(p)-F_\eta^{-1}(p)dp$$
where the right-hand side is not greater than $\int_0^{1-q}F_{\mu_{\underline{\mathcal P}(\nu)}}^{-1}(p)-F^{-1}_\mu(p)dp$ since $\mu_{\underline{\mathcal P}(\nu)}\in\tilde{\mathcal P}(\nu)$.
Since $\eta\in\tilde{\mathcal P}(\nu)$ iff $F_\mu^{-1}-F_\eta^{-1}$ is non-decreasing, $\int_0^1|F_\eta^{-1}(p)|dp<\infty$, $\int_0^1F_\eta^{-1}(p)dp=\int_0^1F_\nu^{-1}(p)dp$ and for all $q\in[0,1]$, $\int_0^{1-q}F_\eta^{-1}(p)dp\ge\int_0^{1-q}F_\nu^{-1}(p)dp$ (see \eqref{caraccxquant}), the definition of $\psi$ implies that for all $q\in[0,1]$, $\sup_{\eta\in\tilde{\mathcal P}(\nu)}\int_0^{1-q}F^{-1}_\mu(p)-F_\eta^{-1}(p)dp\le \psi(1-q)=\int_0^{1-q}F^{-1}_\mu(p)-F_{\mu_{\underline{\mathcal P}(\nu)}}^{-1}(p)dp$. Hence $\int_q^1F_{\pi}^{-1}(p)dp=\int_0^{1-q}F_{\mu_{\underline{\mathcal P}(\nu)}}^{-1}(p)-F_{\mu}^{-1}(p)dp=\int_q^1F_{D(\mu_{\underline{\mathcal P}(\nu)})}^{-1}(p)dp$ for all $q\in[0,1]$ which ensures that $\pi$ is the distribution $D(\mu_{\underline{\mathcal P}(\nu)})$ of $F_{\mu_{\underline{\mathcal P}(\nu)}}^{-1}(1-U)-F_\mu^{-1}(1-U)$. Therefore, if $\rho>1$ is such that $\mu,\nu\in{\mathcal P}_\rho(\R)$, 
\begin{align*}
   W_\rho^\rho(\mu,\mu_{\underline{\mathcal P}(\nu)})&=\E\left[|F_{\mu_{\underline{\mathcal P}(\nu)}}^{-1}(1-U)-F_\mu^{-1}(1-U)|^\rho\right]=\int_{\R}|x|^\rho\pi(dx)\\&\le \inf_{\eta\in \tilde{\mathcal P}(\nu)}\E\left[|F_\mu^{-1}(1-U)-F_\eta^{-1}(1-U)|^\rho\right]=\inf_{\eta\in \tilde{\mathcal P}(\nu)}W_\rho^\rho(\mu,\eta)=\inf_{\eta\in\underline{\mathcal P}(\nu)}W_\rho^\rho(\mu,\eta),
\end{align*}
where we used the definition of $\pi$ and the convexity of $\R\ni x\mapsto|x|^\rho$ for the inequality and Lemma~\ref{lem_couplage_R} for the final equality.
Since, by Theorem~\ref{lem_couplage_Rd_gen},  $\mu^\rho_{\underline{\mathcal P}(\nu)}$ is the unique minimizer of $W_\rho^\rho(\mu,\eta)$ on $\underline{\mathcal P}(\nu)$, we conclude that $\mu_{\underline{\mathcal P}(\nu)}=\mu^\rho_{\underline{\mathcal P}(\nu)}$.

From the left-continuity of the quantile functions, we get $F_{\mu_{\underline{\mathcal P}(\nu)}}^{-1}(p)=F_{\mu}^{-1}(p)-\psi'(p-)$ for $p\in(0,1)$, and this function is non-decreasing. Thus, $T$ is nondecreasing. To conclude the proof, it is now sufficient to check that $T(F_{\mu}^{-1}(p))=F_\mu^{-1}(p)-\psi'(p-)$ for a.e. $p\in(0,1)$. Indeed, combined with the inverse transform sampling and Proposition 2.17~\cite{santambrogio}, this ensures that $T\#\mu=\mu_{\underline{\mathcal P}(\nu)}$ and
$$W_\rho^\rho(\mu,\mu_{\underline{\mathcal P}(\nu)})=\int_0^1|F_\mu^{-1}(p)-T(F_\mu^{-1}(p))|^\rho dp=\int_\R|x-T(x)|^\rho\mu(dx).$$
By definition of the quantile function $F_\mu^{-1}$, for all $x\in\R$, $F_\mu^{-1}(F_\mu(x))\le x$ and by right-continuity of $F_\mu$, for all $p\in (0,1)$, $F_\mu(F_\mu^{-1}(p))\ge p$. With the monotonicity of $F_\mu$ we deduce that for all $x\in\R$ such that $F_\mu(x)\in (0,1)$, $F_\mu(F_\mu^{-1}(F_\mu(x)))=F_\mu(x)$. Therefore, if $p\in (0,1)$ is such that $p=F_\mu(x)$ for some $x\in \R$, then $T(F_{\mu}^{-1}(p))=F_\mu^{-1}(p)-\psi'(p-)$. Otherwise, $p\in [F_\mu(x-),F_\mu(x))$ for some $x\in\R$ such that $\mu(\{x\})>0$. We observe that $F_\mu^{-1}(q)$ and $\psi'(q-)$ are constant on $(F_\mu(x-),F_\mu(x)]$ since $q\mapsto \int_0^q F_\mu^{-1}(u)-F_\nu^{-1}(u)du$ is concave on this interval. For $p\in (F_\mu(x-),F_\mu(x)]$, we have $F_\mu^{-1}(p)= x$ and we get $T(F_\mu^{-1}(p))=F_\mu^{-1}(F_\mu(x))-\psi'(F_\mu(x)-)=F_\mu^{-1}(p)-\psi'(p-)$. Therefore the equality holds for $p$ outside the countable set $\{F_\mu(x-):x\in\R\mbox{ s.t. }\mu(\{x\})>0\}$.
\end{proof}

\section{Approximations in the convex order}\label{convexapproxmultidim}
The next proposition is the key result to construct approximations of probability measures that preserve the convex order.

\begin{proposition}\label{prop_cv_curlywedge}
  Let $\rho\ge 1$, $\mu,\nu,\mu_I,\nu_J\in \mathcal{P}_\rho(\R^d)$ such that $\mu \lecx \nu$. Then, we have
  $$W_\rho(\mu,(\mu_I)^\rho_{\underline{\mathcal P}(\nu_J)}) \le 2W_\rho(\mu,\mu_I)+W_\rho(\nu,\nu_J),$$
where, for $\rho=1$, by a slight abuse of notation, $(\mu_I)^1_{\underline{\mathcal P}(\nu_J)}$ denotes any $\eta_\star\in\underline{\mathcal P}(\nu_J)$ such that $W_1(\mu_I,\eta_\star)=\inf_{\eta\in\underline{\mathcal P}(\nu_J)}W_1(\mu_I,\eta)$.\end{proposition}
Let $\mu,\nu \in \cP_\rho(\R^d)$ be such that $\mu \lecx \nu$. From Proposition~\ref{prop_cv_curlywedge},  if we have approximations $\mu_I$ and $\nu_J$ that satisfy $W_\rho(\mu,\mu_I)\underset{I\rightarrow +\infty}\rightarrow 0$ and $W_\rho(\nu,\nu_J)\underset{J\rightarrow +\infty}\rightarrow 0$, then $(\mu_I)^\rho_{\underline{\mathcal P}(\nu_J)}$ also approximates $\mu$ since we have $W_\rho(\mu,(\mu_I)^\rho_{\underline{\mathcal P}(\nu_J)})\underset{I,J\rightarrow +\infty}\rightarrow 0$. In particular, if we take i.i.d. samples $(X_i)_{i\ge 1}$ (resp. $(Y_j)_{j\ge 1}$) distributed according to~$\mu$ (resp. $\nu$), the empirical measure $\mu_I=\frac 1I \sum_{i=1}^I\delta_{X_i}$ (resp. $\nu_J=\frac 1J \sum_{j=1}^J\delta_{Y_j}$) satisfy $W_\rho(\mu,\mu_I)\underset{I\rightarrow +\infty}\rightarrow 0$ (resp. $W_\rho(\nu,\nu_J)\underset{J\rightarrow +\infty}\rightarrow 0$) almost surely. Indeed, the law of large numbers gives the almost sure weak convergence of $\mu_I$ towards $\mu$ as well as the almost sure convergence of $\frac{1}{I}\sum_{i=1}^I |X_i|^\rho$ to $\int_{\R^d}|x|^\rho \mu(dx)$. By Proposition~7.1.5 of~\cite{AGS}, we get $W_\rho(\mu,\mu_I)\underset{I\rightarrow +\infty}\rightarrow 0$ almost surely. Under more restrictive assumptions on the measures $\mu$ and $\nu$, we can have almost sure estimates on the rate of convergence. Let us assume that $\mu$ is such that $\mathcal{E}_{\alpha,\gamma}=\int_{\R^d} e^{\gamma|x|^\alpha} \mu(dx)<\infty$ for some $\alpha>\rho$ and $\gamma>0$. Then, by  Theorem~2 of Fournier and Guillin~\cite{FoGu}, there are constants $c,C>0$ depending on $\rho,d,\alpha,\gamma,\mathcal{E}_{\alpha,\gamma}$ such that
$$\forall x\in (0,1), \P(W_\rho(\mu,\mu_I)>x)=\P(W^\rho_\rho(\mu,\mu_I)>x^\rho)\le C\exp(-c I x^{d\vee(2\rho)}).$$
Therefore we have $\sum_{I=2}^\infty \P\left(W_\rho(\mu,\mu_I)> \left(\frac{2\log(I)}{cI}  \right)^{\frac{1}{d\vee(2\rho)}} \right)\le C\sum_{I=2}^\infty I^{-2}<\infty$, which gives that almost surely, there exists $I_0$ such that $\forall I \ge I_0, W_\rho(\mu,\mu_I)\le \left(\frac{2\log(I)}{cI}  \right)^{\frac{1}{d\vee(2\rho)}}$. Since $x \mapsto e^{\gamma|x|^\alpha}$ is convex, $\int_{\R^d} e^{\gamma|x|^\alpha} \nu(dx)<\infty \implies \int_{\R^d} e^{\gamma|x|^\alpha} \mu(dx)<\infty$, in which case we have both $W_\rho(\mu,\mu_I)=\mathcal{O}\left(\left(\frac{\log(I)}{I}  \right)^{\frac{1}{d\vee(2\rho)}} \right)$ and $W_\rho(\nu,\nu_J)=\mathcal{O}\left(\left(\frac{\log(J)}{J}  \right)^{\frac{1}{d\vee(2\rho)}} \right)$ and thus $$W_\rho(\mu,(\mu_{I})_{\underline{\mathcal P}(\nu_J)}^\rho)\underset{I,J\rightarrow +\infty}=\mathcal{O}\left(\left(\frac{\log(I\wedge J)}{I \wedge J}  \right)^{\frac{1}{d\vee(2\rho)}} \right), a.s.$$ 
Theorem~2 of~\cite{FoGu} also gives upper bounds of $\P(W_\rho(\mu,\mu_I)>x)$ under different weaker assumptions on~$\mu$. We can repeat the same argument in those cases and get a weaker rate of convergence of  $W_\rho(\mu,\mu_I)$ towards~$0$.

We now briefly consider the multi-marginal case.  Let $\rho\ge 1$, $\ell\ge 2$, $I_1,\hdots,I_\ell$ be positive integers and $\mu^1,\dots,\mu^\ell$ be probability measures on $\R^d$  such that $\mu^1\lecx \dots \lecx \mu^\ell$ and $\int_{\R^d} |x|^\rho\mu^\ell(dx) <\infty$. We consider for $1\le k \le \ell$, $\mu^k_{I_k}=\frac{1}{I_k}\sum_{i=1}^{I_k} \delta_{X^k_i}$ the empirical measure of an i.i.d. sample $X^k_1,\dots,X^k_{I_k}$ distributed according to~$\mu^k$. Let us set $\mu^{\ell,\rho}_{I_\ell}=\mu^{\ell}_{I_\ell}$ and define (using for $\rho=1$ the abuse of notation made in Proposition~\ref{prop_cv_curlywedge}) by backward induction for $k\in\{1,\hdots,\ell-1\}$,  the projection $\mu^{k,\rho}_{I_k,\hdots,I_{\ell} }$ of $\mu^{k}_{I_k}$ on the set $\underline{\mathcal P}(\mu^{k+1,\rho}_{I_{k+1},\hdots,I_\ell})$ for the $W_\rho$-Wasserstein distance. Then, by Proposition~\ref{prop_cv_curlywedge}, we have for $1\le k \le \ell-1$,
$$W_\rho(\mu^k,\mu^{k,\rho}_{I_k,\hdots,I_{\ell} })\le 2W_\rho(\mu^k,\mu^k_{I_k})+ W_\rho(\mu^{k+1},\mu^{k+1,\rho}_{I_{k+1},\hdots,I_{\ell} }).$$
Therefore, we deduce by induction that
$$W_\rho(\mu^k,\mu^{k,\rho}_{I_k,\hdots,I_{\ell} })\le 2 \sum_{k'=k}^{\ell-1}W_\rho(\mu^{k'},\mu^{k'}_{I_{k'}})+ W_\rho(\mu^{\ell},\mu^{\ell}_{I_{\ell}}).$$
We eventually get the following result. 
\begin{proposition}\label{propconv2}
 Let $\rho\ge 1$, $\mu^1,\dots,\mu^\ell$ be probability measures on $\R^d$ such that $\mu^1\lecx \dots \lecx \mu^\ell$ and $\int_{\R^d} |x|^\rho\mu^\ell(dx) <\infty$. 
 Then, as $I_1,\dots,I_\ell \rightarrow + \infty$, $\sum_{k=1}^\ell W_\rho(\mu^k, \mu^{k,\rho}_{I_k,\hdots,I_{\ell} })$ converges almost surely to~$0$. Besides, if $\int_{\R^d} e^{\gamma|x|^\alpha} \mu^\ell(dx)$ for some $\alpha>\rho$ and $\gamma>0$, we have a.s. $\sum_{k=1}^\ell W_\rho(\mu^k,\mu^{k,\rho}_{I_k,\hdots,I_{\ell} }) \underset{\min_{k=1,\dots,\ell} I_k \rightarrow +\infty}{=}\mathcal{O}\left(\left(\frac{\log( \min_{k=1,\dots,\ell} I_k)}{\min_{k=1,\dots,\ell} I_k}  \right)^{\frac{1}{d\vee(2\rho)}} \right)$.
\end{proposition}

\begin{proof}[Proof of Proposition~\ref{prop_cv_curlywedge}]
  We consider $\rho>1$. Let $Q^\rho_{\mu_I}$ (resp. $Q^\rho_\nu$) be a Markov kernel such that $\mu_I(dx)Q^\rho_{\mu_I}(x,dy)$ (resp. $\nu(dx)Q^\rho_\nu(x,dy)$) is an optimal transport plan for $W_\rho(\mu_I,\mu)$  (resp. $W_\rho(\nu,\nu_J)$). Let $R(x,dy)$ be a martingale kernel such that $\nu=\mu R$. We observe that $Q^\rho_{\mu_I} R Q^\rho_\nu$ is a Markov kernel such that $\mu_IQ^\rho_{\mu_I} R Q^\rho_\nu=\mu R Q^\rho_\nu=\nu Q^\rho_\nu=\nu_J$. By Theorem~\ref{lem_couplage_Rd_gen}, then using the martingale property of $R$, the Jensen and Minkowski inequalities, we get
  \begin{align*}
   & W_\rho(\mu_I, (\mu_I)^\rho_{\underline{\mathcal P}(\nu_J)}) \le \mathcal{J}_\rho^{1/\rho} (Q^\rho_{\mu_I} R Q^\rho_\nu) \\
    &=\left(\int_{\R^d}\left| \int_{\R^d\times\R^d\times\R^d}(x-w+w-y)Q^\rho_{\mu_I}(x,dw) R(w,dz)Q^\rho_\nu(z,dy)\right|^\rho\mu_I(dx)\right)^{1/\rho}
    \\&=\left(\int_{\R^d}\left| \int_{\R^d\times\R^d\times\R^d}(x-w+z-y)Q^\rho_{\mu_I}(x,dw) R(w,dz)Q^\rho_\nu(z,dy)\right|^\rho\mu_I(dx)\right)^{1/\rho}
    \\
    &\le\left(\int_{\R^d\times\R^d\times\R^d\times\R^d}\left|x-w+z-y\right|^\rho Q^\rho_{\mu_I}(x,dw) R(w,dz)Q^\rho_\nu(z,dy)\mu_I(dx)\right)^{1/\rho}
  \\
  &\le \left(\int_{\R^d\times\R^d}|x-w|^\rho Q^\rho_{\mu_I}(x,dw)\mu_I(dx)\right)^{1/\rho}+\left(\int_{\R^d\times\R^d}|z-y|^\rho \nu(dz)Q^\rho_\nu(z,dy)\right)^{1/\rho} \\
&= W_\rho(\mu_I,\mu)+W_\rho(\nu_J,\nu).
  \end{align*}
The claim follows since $W_\rho(\mu, (\mu_I)^\rho_{\underline{\mathcal P}(\nu_J)})\le W_\rho(\mu, \mu_I)+W_\rho(\mu_I, (\mu_I)^\rho_{\underline{\mathcal P}(\nu_J)})$.  
\end{proof}

\section{Wasserstein projection of $\nu$ on the set of probability measures larger than $\mu$ in the convex order}\label{wasprojnu}

Let $\mu,\nu\in \cP_\rho(\R^d)$. We have just presented  a construction of a measure $\mu^\rho_{\underline{\mathcal P}(\nu)}$ such that $\mu^\rho_{\underline{\mathcal P}(\nu)} \lecx \nu$. Then,  a natural question is: can we construct similarly a measure $\nu^\rho_{\bar{\mathcal P}(\mu)}$ such that $\mu \lecx \nu^\rho_{\bar{\mathcal P}(\mu)}  $? Let us start again with two empirical measures $\mu^I=\frac 1I \sum_{i=1}^I \delta_{X_i}$ and $\nu^J=\frac 1J \sum_{j=1}^J \delta_{Y_j}$. A natural construction would be to take  $(\nu_J)^\rho_{\bar{\mathcal P}(\mu_I)} =\frac 1J \sum_{j=1}^J \delta_{\tilde{Y}_j}$, where $(\tilde{Y}_j,j=1,\dots,J)\in (\R^d)^J$ minimizes $\sum_{j=1}^J|\tilde{Y}_j-Y_j|^\rho$ under the constraint $\mu^I \lecx \frac 1J \sum_{j=1}^J \delta_{\tilde{Y}_j}$ (this constraint can always be satisfied  when $J=I$ by taking $\tilde{Y}_j=X_j$ for $j=1,\dots,J$ or when $J\ge d+1$ by taking $\tilde{Y}_j,\ j=1,\dots,d+1$ as the images of the vertices of the canonical simplex by some similarity transformation). The analogous construction for $\nu^\rho_{\bar{\mathcal P}(\mu)} $ would be then to take  $\nu^\rho_{\bar{\mathcal P}(\mu)} =T \# \nu$, where $T:\R^d\rightarrow \R^d$ is a measurable map that minimizes $\int_{\R^d}|y-T(y)|^\rho\nu(dy)$, under the constraint $\mu \lecx T \# \nu$. More generally, we define
  $$\nu^\rho_{\bar{\mathcal P}(\mu)}:= \arg \min_{\eta\in\bar{\mathcal P}(\mu)}W_\rho(\nu,\eta)\mbox{ where }\bar{\mathcal P}(\mu)=\{\eta\in{\mathcal P}(\R^d):\mu\lecx\eta\}.$$
Let us now assume that $\rho>1$. The latter problem coincides with the former one when $\nu$ is absolutely continuous with respect to the Lebesgue measure (i.e. $\nu(A)=0$ for any Borel set~$A$ with zero Lebesgue measure), since we know in this case that the optimal coupling for the Wasserstein distance $W_\rho$ is given by a transport map, see e.g. Theorem~6.2.4 in~\cite{AGS}. We now check that it is well defined. Let $(\eta_n)_{n\ge 1}\in (\cP_\rho(\R^d))^\N$ be such that $\eta_n\in\bar{\mathcal P}(\mu)$ and $W_\rho(\nu,\eta_n)\underset{n\rightarrow +\infty}{\rightarrow}\inf_{\eta\in\bar{\mathcal P}(\mu)}W_\rho(\nu,\eta)$. Let $\pi_n\in\Pi(\nu,\eta_n)$ denote an optimal transport plan between $\nu$ and $\eta_n$ for $W_\rho$. We have $\left(\int |x|^\rho\eta_n(x)\right)^{1/\rho}=W_{\rho}(\eta_n,\delta_0)\le W_{\rho}(\eta_n,\nu)+ W_{\rho}(\nu,\delta_0)$: the boundedness of the moments ensures that there is a subsequence such that $\pi_{\varphi(n)}$ and $\eta_{\varphi(n)}$ weakly converges to~$\pi_\infty$ and $\eta_\infty$. This gives $\inf_{\eta\in\bar{\mathcal P}(\mu)}W^\rho_\rho(\nu,\eta)\ge \lim_{n\rightarrow+\infty} \int (|x-y|^\rho\wedge K) \pi_{\varphi(n)}(dx,dy)=\int (|x-y|^\rho\wedge K) \pi_{\infty}(dx,dy)$ for any $K>0$. By monotone convergence, we deduce that $ \inf_{\eta\in\bar{\mathcal P}(\mu)}W_\rho(\nu,\eta)\ge \int |x-y|^\rho \pi_{\infty}(dx,dy)$. Clearly, $\pi_\infty$ is a coupling between $\nu$ et $\eta_\infty$. Besides, from the uniform integrability given by the bounds on the $\rho$-th moment, we get that for any convex function $\phi:\R^d\rightarrow \R^d$ such that $\sup_{x\in\R^d}\frac{|\phi(x)|}{1+|x|}<\infty$, $\int \phi(x)\mu(dx) \le \int \phi(x)\eta_{\varphi(n)}(dx) \underset{n\rightarrow+\infty}{\rightarrow}  \int \phi(x)\eta_{\infty}(dx)$. Therefore, by Lemma~\ref{lem_ordre_cvx} below, $\eta_\infty\in\bar{\mathcal P}(\mu)$, which shows the existence of a minimum. When $\nu$ is absolutely continuous with respect to the Lebesgue measure, we can show that this minimum is unique. Let us consider $\eta_1,\eta_2\in\bar{\mathcal P}(\mu)$ such that $W_\rho(\nu,\eta_1)=W_\rho(\nu,\eta_2)=\inf_{\eta\in\bar{\mathcal P}(\mu)}W_\rho(\nu,\eta)$. One has $\frac12 \left( \eta_1+ \eta_2\right)\in\bar{\mathcal P}(\mu)$, and, by Lemma~\ref{lem_wass_strict} below, we get $W_\rho\left(\nu, \frac12 \left( \eta_1+ \eta_2\right)\right)\le\inf_{\eta\in\bar{\mathcal P}(\mu)}W_\rho(\nu,\eta)$ and $\eta_1=\eta_2$ since the inequality is necessarily an equality. In dimension $1$, uniqueness still holds without any assumption on $\nu$. Indeed, by \eqref{caraccxquant}, the probability measure $\bar{\eta}_{12}$ defined by $F_{\bar{\eta}_{12}}^{-1}=\frac{1}{2}(F_{\eta_1}^{-1}+F_{\eta_2}^{-1})$ is such that $\mu\lecx\bar{\eta}_{12}$. Again by Lemma~\ref{lem_wass_strict},  $W_\rho\left(\nu, \bar{\eta}_{12}\right)\le\inf_{\eta\in\bar{\mathcal P}(\mu)}W_\rho(\nu,\eta)$ and $\eta_1=\eta_2$ since the inequality is necessarily an equality.
In dimension $d=1$, if $\mu,\nu\in{\mathcal P}_1(\R)$, let $$\tilde{\mathcal P}(\mu):=\{\eta\in\bar{\mathcal P}(\mu)\cap{\mathcal P}_1(\R):(0,1)\ni p\mapsto F_\eta^{-1}(p)-F_\nu^{-1}(p)\mbox{ is non-decreasing}\}.$$ 
Let $\tilde \psi$ denote the concave hull (smallest concave function larger than) of the function $q\mapsto \int_q^1F_\mu^{-1}(p)-F_\nu^{-1}(p)dp$. There is a probability measure $\nu_{\bar{\mathcal P}(\mu)}$ such that
$\int_q^1F_{\nu_{\bar{\mathcal P}(\mu)}}^{-1}(p) dp=\tilde \psi(q)+\int_q^1F_\nu^{-1}(p)dp$. Moreover, $\nu_{\bar{\mathcal P}(\mu)}\in\tilde{\mathcal P}(\mu)$. For $\eta\in \tilde{\mathcal P}(\mu)$, let $D(\eta)$ denote the distribution of $F_\eta^{-1}(U)-F_\nu^{-1}(U)$ for $U$ uniformly distributed on $(0,1)$. By Lemma~\ref{leminf} below, the set $\{D(\eta):\eta\in \tilde{\mathcal P}(\mu)\}$ admits an infimum $\pi$ for the convex order and for all $q\in[0,1]$, $\int_q^1F_\pi^{-1}(p)dp=\inf_{\eta\in \tilde{\mathcal P}(\mu)}\int_q^1F_{D(\eta)}^{-1}(p)$. For $\eta\in \tilde{\mathcal P}(\mu)$, one has $F_{D(\eta)}^{-1}=F_\eta^{-1}-F_\nu^{-1}$ by Lemma~\ref{lemquant} below. With the fact that $\eta\in \tilde{\mathcal P}(\mu)$ if and only if $\int_q^1F_\eta^{-1}(p)dp\ge\int_q^1F_\mu^{-1}(p)dp$ for all $q\in(0,1)$ with equality for $q=0$ and $[0,1]\ni q\mapsto \int_q^1F_\eta^{-1}(p)-F_\nu^{-1}(p)dp$ is concave, one deduces that for $q\in(0,1)$,
\begin{align*}
   \int_q^1F_\pi^{-1}(p)dp=\inf_{\eta\in \tilde{\mathcal P}(\mu)}\int_q^1F_\eta^{-1}(p)-F_\nu^{-1}(p)dp=\tilde \psi(q)=\int_q^1F_{\nu_{\bar{\mathcal P}(\mu)}}^{-1}(p)-F_\nu^{-1}(p)dp.
\end{align*}
Hence $\pi=D(\nu_{\bar{\mathcal P}(\mu)})$. If $\mu,\nu\in{\mathcal P}_\rho(\R)$ for some $\rho>1$, then
\begin{align*}
   W_\rho^\rho(\nu,\nu_{\bar{\mathcal P}(\mu)})&=\E\left[|F_{\nu_{\bar{\mathcal P}(\mu)}}^{-1}(U)-F_\nu^{-1}(U)
|^\rho\right]=\int_\R|x|^\rho\pi(dx)\\&\le \inf_{\eta\in\tilde{\mathcal P}(\mu)}\E\left[|F_{\eta}^{-1}(U)-F_\nu^{-1}(U)|^\rho\right]=\inf_{\eta\in\tilde{\mathcal P}(\mu)}W_\rho^\rho(\nu,\eta).\end{align*}

By Lemma~\ref{lemcrois} below, $\inf_{\eta\in\tilde{\mathcal P}(\mu)}W_\rho(\nu,\eta)=\inf_{\eta\in\bar{\mathcal P}(\mu)}W_\rho(\nu,\eta)$. Therefore $W_\rho(\nu,\nu_{\bar{\mathcal P}(\mu)})=\inf_{\eta\in\bar{\mathcal P}(\mu)}W_\rho(\nu,\eta)$ and $\nu^\rho_{\bar{\mathcal P}(\mu)}=\nu_{\bar{\mathcal P}(\mu)}$.

For probability measures $\mu_I=\sum_{i=1}^Ip_i\delta_{x_i}$ (resp. $\nu_J=\sum_{j=1}^Jq_j\delta_{y_j}$) on the real line with $(p_1,\hdots,p_I)\in(0,1]^I$ and $x_1<x_2<\hdots<x_I$ (resp. $(q_1,\hdots,q_J)\in(0,1]^J$ and $y_1<y_2<\hdots<y_J$), $\tilde\psi$ is equal to $\int_0^1F_{\mu_I}^{-1}(p)-F_{\nu_J}^{-1}(p)dp$ minus the convex hull $\psi$ of $q\mapsto \int_0^qF_{\mu_I}^{-1}(p)-F_{\nu_J}^{-1}(p)dp$ which has already been discussed after Theorem~\ref{propprojd1} and can be computed by Andrew's monotone chain algorithm. One then may compute the probability measure $(\nu_J)_{\bar{\mathcal P}(\mu_I)}$ which writes $\sum_{k=1}^K r_k\delta_{z_k}$ with $K\le I+J$, $z_1\le z_2\le\hdots\le z_K$ and $(r_k)_{1\le k\le K}$ denoting the differences between the successive elements of the increasing reordering of $\{0\}\cup\{\sum_{k=1}^ip_k:1\le i\le I\}\cup\{\sum_{k=1}^jq_k:1\le j\le J\}$.

Let $\mu,\nu\in \cP_{\rho}(\R^d)$ such that $\mu\lecx \nu$ and $\mu_I,\nu_J\in \cP_{\rho}(\R^d)$ be arbitrary approximations of $\mu$ and $\nu$.  The probability measure $(\nu_J)_{\bar{\mathcal P}(\mu_I)}^\rho$ (or any minimizing probability measure when uniqueness is not shown) satisfies
\begin{equation}\label{cv_curlyvee}
  W_\rho((\nu_J)_{\bar{\mathcal P}(\mu_I)}^\rho,\nu)\le W_\rho(\mu,\mu_I)+2W_\rho(\nu,\nu_J)
\end{equation}
  We proceed like in the proof of Proposition~\ref{prop_cv_curlywedge}. Let $Q^\rho_{\mu_I}$ (resp. $Q^\rho_\nu$) be a Markov kernel such that $\mu_I(dx)Q^\rho_{\mu_I}(x,dy)$ (resp. $\nu(dx)Q^\rho_\nu(x,dy)$) is an optimal transport plan for $W_\rho(\mu_I,\mu)$  (resp. $W_\rho(\nu,\nu_J)$) and $R$ be a martingale kernel such that $\mu R=\nu$. We obviously have  $\nu_J=\mu_IQ^\rho_{\mu_I} R Q^\rho_\nu$. 
By Jensen inequality and using the martingale property of~$R$, we have $\mu_I \lecx ((x,w,z)\mapsto x+z-w) \# \mu_I(dx)Q^\rho_{\mu_I}(x,dw)R(w,dz)$, so that $$\inf_{\eta\in\bar{\mathcal P}(\mu_I)}W_\rho(\nu_J,\eta)\le \left(\int_{(\R^d)^4} |x+z-w-y|^\rho \mu_I(dx)Q^\rho_{\mu_I}(x,dw)R(w,dz)Q^\rho_\nu(z,dy)\right)^{1/\rho}.$$
We get~\eqref{cv_curlyvee} using Minkowski's inequality and the triangle inequality $W_\rho((\nu_J)_{\bar{\mathcal P}(\mu_I)}^\rho,\nu)\le W_\rho((\nu_J)_{\bar{\mathcal P}(\mu_I)}^\rho,\nu_J)+W_\rho(\nu,\nu_J)$. In the multi-marginal case, defining inductively $\mu^{1,\rho}_{I_1}=\mu^1_{I_1}$ and for $k\in\{2,\hdots,\ell\}$, $\mu^{k,\rho}_{I_1,\hdots,I_k}$ as the $W_\rho$ projection of $\mu^k_{I_k}$ on $\bar{\mathcal P}(\mu^{k-1,\rho}_{I_1,\hdots,I_{k-1}})$, we deduce that for $k\in\{2,\hdots,\ell\}$, 
$$W_\rho(\mu^k,\mu^{k,\rho}_{I_1,\hdots,I_k})\le W_\rho(\mu^1,\mu^1_{I_1})+2 \sum_{k'=2}^{k}W_\rho(\mu^{k'},\mu^{k'}_{I_{k'}}).$$

Despite all these interesting properties that we summarize in the next proposition, the measure(s) $\nu^\rho_{\bar{\mathcal P}(\mu)}$ do(es) not seem easy to be calculated numerically, even for $\rho=2$. In fact, the constraint of the convex order is not simple to handle in a minimization program. More precisely, in the case of empirical measures, one would have to minimize $\sum_{j=1}^J|\tilde{Y}_j-Y_j|^2$ under the constraint $ \frac 1I \sum_{i=1}^I \delta_{X_i} \lecx \frac 1J \sum_{j=1}^J \delta_{\tilde{Y}_j}$. Even in dimension~1, this constraint is not linear since it is equivalent to $\max_i X_i \le \max_j \tilde{Y}_j$, $\min_i X_i \le \min_j \tilde{Y}_j$, $\frac 1I \sum_{i=1}^I  X_i =\frac 1J \sum_{j=1}^J \tilde{Y}_j$, and $\frac 1I \sum_{i=1}^I  (X_i-\tilde{Y}_{j'})^+ \le \frac 1J \sum_{j=1}^J (\tilde{Y}_j-\tilde{Y}_{j'})^+$ for any $1\le j'\le J$, see e.g. Corollary 2.2 in~\cite{ACJ1}. This is why we mostly focus on  $\mu^2_{\underline{\mathcal P}(\nu)}$ that leads to a clear implementation of a quadratic problem with linear constraints. 

\begin{theorem}\label{prop_curlyvee}
     For $\rho>1$, if $\mu,\nu \in \mathcal{P}_\rho(\R^d)$, then $\inf_{\eta\in\bar{\mathcal P}(\mu)}W_\rho^\rho(\nu,\eta)$ is attained by some probability measure $\nu^\rho_{\bar{\mathcal P}(\mu)}$ which is unique when $\nu$ is absolutely continuous with respect to the Lebesgue measure or $d=1$. If $\mu,\nu\in{\mathcal P}_1(\R)$, then there is a probability $\nu_{\bar{\mathcal P}(\mu)}$ such that for all $q\in[0,1]$,
$\int_q^1F_{\nu_{\bar{\mathcal P}(\mu)}}^{-1}(p) dp=\tilde\psi(q)+\int_q^1F_\nu^{-1}(p)dp$ where $\tilde \psi$ denotes the concave hull of the function $q\mapsto \int_q^1F_\mu^{-1}(p)-F_\nu^{-1}(p)dp$. Moreover, $\nu^\rho_{\bar{\mathcal P}(\mu)}=\nu_{\bar{\mathcal P}(\mu)}$ for each $\rho>1$ such that $\mu,\nu\in{\mathcal P}_\rho(\R)$.
Last, if $\rho>1$ and $\mu,\nu,\mu_I,\nu_J\in \mathcal{P}_\rho(\R^d)$, then 
  $\mu\lecx\nu\Rightarrow W_\rho((\nu_J)^\rho_{\bar{\mathcal P}(\mu_I)},\nu) \le W_\rho(\mu,\mu_I)+2W_\rho(\nu,\nu_J).$
\end{theorem}
Comparing $W_\rho(\nu^\rho_{\bar{\mathcal P}(\mu)},\nu)$ and $ W_\rho(\mu,\mu^\rho_{\underline{\mathcal P}(\nu)})$ leads to interesting properties.
\begin{corollary}\label{cor_equality_wass} For $\rho>1$, $\mu,\nu\in\mathcal{P}_\rho(\R^d)$, we have
  $$ W_\rho(\nu^\rho_{\bar{\mathcal P}(\mu)},\nu) = W_\rho(\mu,\mu^\rho_{\underline{\mathcal P}(\nu)}) $$
and there is a measurable transport map $T:\R^d\to\R^d$ such that the only optimal transport plan between $\nu^\rho_{\bar{\mathcal P}(\mu)}$ and $\nu$ is $\nu^\rho_{\bar{\mathcal P}(\mu)}(dz)\delta_{T(z)}(dy)$. Moreover, for any martingale kernel $R$ such that $\mu R = \nu^\rho_{\bar{\mathcal P}(\mu)}$, $\mu(dx)R(x,dz)$ a.e., $T(z)-z=\int_{\R^d}T(z)R(x,dz)-x$. Last, in dimension~$d=1$, when  $\mu,\nu\in\mathcal{P}_1(\R)$, we also have for all $\rho\ge 1$, $ W_\rho(\nu_{\bar{\mathcal P}(\mu)},\nu) = W_\rho(\mu,\mu_{\underline{\mathcal P}(\nu)})=\left(\int_0^1|\psi'(u-)|^\rho du \right)^{1/\rho}$ where $\psi'(u-)$ is the left-hand derivative of the convex hull $\psi$ of the function $[0,1]\ni q \mapsto \int_0^q F^{-1}_{\mu}(p)-F^{-1}_{\nu}(p) dp$. 
\end{corollary}
\begin{proof}
  Since $\mu\lecx\nu^\rho_{\bar{\mathcal P}(\mu)}$, we may replace $(\mu,\mu_I,\nu,\nu_J)$ by $(\mu,\mu,\nu^\rho_{\bar{\mathcal P}(\mu)},\nu)$ in Proposition~\ref{prop_cv_curlywedge} to get $W_\rho(\mu,\mu^\rho_{\underline{\mathcal P}(\nu)})\le W_\rho(\nu^\rho_{\bar{\mathcal P}(\mu)},\nu)$. Using that $\mu^\rho_{\underline{\mathcal P}(\nu)}\lecx\nu$ to replace $(\mu,\mu_I,\nu,\nu_J)$ by $(\mu^\rho_{\underline{\mathcal P}(\nu)},\mu,\nu,\nu)$ in Theorem~\ref{prop_curlyvee}, we obtain the converse inequality.

  Now, let $\rho>1$, $R$ denote a martingale kernel such that $\mu R = \nu^\rho_{\bar{\mathcal P}(\mu)}$ and $Q$ a Markov kernel such that $\nu^\rho_{\bar{\mathcal P}(\mu)}(dz)Q(z,dy)$ is  an optimal transport plan for $W_\rho(\nu^\rho_{\bar{\mathcal P}(\mu)},\nu)$. Repeating the arguments of Proposition~\ref{prop_cv_curlywedge} (replacing again $(\mu,\mu_I,\nu,\nu_J)$ by $(\mu,\mu,\nu^\rho_{\bar{\mathcal P}(\mu)},\nu)$), we get
  \begin{align*}
    &W^\rho_\rho(\mu,\mu^\rho_{\underline{\mathcal P}(\nu)})\le \int_{\R^d} \left| \int_{\R^d\times \R^d} (x-y) R(x,dz)Q(z,dy)\right|^\rho \mu(dx)\\&= \int_{\R^d} \left| \int_{\R^d\times \R^d} (z-y) R(x,dz)Q(z,dy)\right|^\rho \mu(dx)\le \int_{\R^d}\left| \int_{\R^d} (z-y) Q(z,dy)\right|^\rho \nu^\rho_{\bar{\mathcal P}(\mu)}(dz)\\
    & \le 
 \int_{\R^d\times \R^d} |z-y|^\rho  \nu^\rho_{\bar{\mathcal P}(\mu)}(dz)Q(z,dy) =W^\rho_\rho(\nu^\rho_{\bar{\mathcal P}(\mu)},\nu)=W^\rho_\rho(\mu,\mu^\rho_{\underline{\mathcal P}(\nu)}).
  \end{align*}
The equality in the last inequality ensures that  $\nu^\rho_{\bar{\mathcal P}(\mu)}(dz)$ a.e., $Q(z,dy)=\delta_{T(z)}(dy)$ where $T(z)=\int_{\R^d}yQ(z,dy)$. Moreover, the equality in the second inequality implies that $\mu(dx)R(x,dz)$ a.e., $T(z)-z=\int_{\R^d}T(z)R(x,dz)-x$.

If $\tilde Q$ is another Markov kernel such that $\nu^\rho_{\bar{\mathcal P}(\mu)}(dz)\tilde Q(z,dy)$ is  an optimal transport plan for $W_\rho(\nu^\rho_{\bar{\mathcal P}(\mu)},\nu)$, then $\nu^\rho_{\bar{\mathcal P}(\mu)}(dz)\frac{\tilde Q+Q}{2}(z,dy)$ is also an optimal transport plan and $\nu^\rho_{\bar{\mathcal P}(\mu)}(dz)$ a.e., $\frac{\tilde Q+Q}{2}(z,dy)$ is a Dirac mass so that $\tilde Q(z,dy)=Q(z,dy)$.

In dimension~$1$, we observe that
\begin{equation}\label{def_curlies}
\forall q\in [0,1],\ \int_0^qF_{\mu_{\underline{\mathcal P}(\nu)}}^{-1}(p)dp=\int_0^qF_{\mu}^{-1}(p)dp-\psi(q)\text{ and }\int_0^qF_{\nu_{\bar{\mathcal P}(\mu)}}^{-1}(p)dp=\int_0^q F_{\nu}^{-1}(p)dp +\psi(q).
\end{equation}
Thus, we have $F_{\mu_{\underline{\mathcal P}(\nu)}}^{-1}(p)-F_{\mu}^{-1}(p)=-\psi'(p-)$ and $F_{\nu_{\bar{\mathcal P}(\mu)}}^{-1}(p)-F_{\nu}^{-1}(p)=\psi'(p-)$ for $p\in (0,1)$, which gives the claim.
\end{proof}

The property $T(z)-z=\int_{\R^d}T(z)R(x,dz)-x$, $\mu(dx)R(x,dz)$ a.e., in Corollary~\ref{cor_equality_wass} indicates that in dimension~1, an optimal transport map $T$ between $\nu_{\bar{\mathcal P}(\mu)}$ and $\nu$ should be piecewise affine with slope~1 on the irreducible components of $(\mu, \nu_{\bar{\mathcal P}(\mu)})$ introduced in Theorem A.4~\cite{BeJu}, provided that we can find a martingale kernel~$R$ that spans the whole components. This is indeed the case according to the following proposition which moreover exhibits a common optimal transport map for $W_\rho(\nu_{\bar{\mathcal P}(\mu)},\nu)$ and $W_\rho(\mu ,\mu_{\underline{\mathcal P}(\nu)})$.
\begin{proposition} Let $\rho>1$, $\mu,\nu \in \mathcal{P}_\rho(\R)$. Let $(\ut_n,\ot_n)$, $1\le n \le  N$, (resp. $(\ut'_n,\ot'_n)$, $1\le n \le  N'$)  be the irreducible components of $(\mu, \nu_{\bar{\mathcal P}(\mu)})$ (resp. $(\mu_{\underline{\mathcal P}(\nu)},\nu)$). Then, we have $N=N'$ and $F_{\mu }(\ut_n)=F_{\mu_{\underline{\mathcal P}(\nu)}}(\ut'_n)$, $F_{\mu }(\ot_n-)=F_{\mu_{\underline{\mathcal P}(\nu)}}(\ot'_n-)$ up to a renumbering of $(t'_n)_{1\le n\le N }$.

  Let $\psi$  be the convex hull of the function $[0,1]\ni q \mapsto \int_0^q F^{-1}_{\mu}(p)-F^{-1}_{\nu}(p) dp$. Then, the function $T:\R\rightarrow \R$ defined by \begin{align*}&\forall x\notin \cup_{1\le n \le N} (\ut_n,\ot_n),\;T(x)=F_\nu^{-1}(F_{\nu_{\bar{\mathcal P}(\mu)}}(x))\mbox{ and }\\
&\forall 1\le n\le N,\;\forall x\in (\ut_n,\ot_n),\;T(x)=x-\frac{\psi(F_{\mu }(\ot_n-))-\psi(F_{\mu }(\ut_n))}{F_{\mu }(\ot_n-)-F_{\mu }(\ut_n)}\end{align*}
is an optimal transport map for $W_\rho(\nu_{\bar{\mathcal P}(\mu)},\nu)$ and $W_\rho(\mu ,\mu_{\underline{\mathcal P}(\nu)})$.
\end{proposition}
\begin{proof} We set $\uq_n=F_{\mu }(\ut_n)$ and $\oq_n=F_{\mu }(\ot_n-)$.  From \eqref{def_curlies} and Lemma~\ref{prop_irred_comp2} below which characterizes the irreducible components in terms of the quantile functions, we get
  \begin{align*}
    \bigcup_{1\le n \le  N}(\uq_n,\oq_n) &=    \left\{q \in [0,1],  \int_0^q F_{\mu }^{-1} (p)dp > \int_0^q F_{\nu_{\bar{\mathcal P}(\mu)}}^{-1} (p)dp \right \} \\ 
    &=\left\{q \in [0,1], \psi(q)< \int_0^q F_{\mu }^{-1} (p)dp - \int_0^q F_{\nu}^{-1} (p)dp \right \}\\
    &= \left\{q \in [0,1], \int_0^q F_{\mu_{\underline{\mathcal P}(\nu)}}^{-1} (p)dp> \int_0^q F_{\nu}^{-1} (p)dp \right\}\\
    &=  \bigcup_{1\le n \le  N'}(F_{\mu_{\underline{\mathcal P}(\nu)}}(\ut'_n), F_{\mu_{\underline{\mathcal P}(\nu)} }(\ot'_n-) ),
  \end{align*}
  which gives the first claim. From the second equality and since $\psi$ is the convex hull of $[0,1]\ni q\mapsto\int_0^qF_\mu^{-1}(p)-F_\nu^{-1}(p)dp$, we get 
\begin{equation}
 \forall q\notin \cup_{1\le n\le N} (\uq_n,\oq_n),\;\psi(q)= \int_0^q F_{\mu }^{-1} (p)dp - \int_0^q F_{\nu}^{-1} (p)dp\label{egalfconvf}
\end{equation}
and 
$\psi(q)=\psi(\uq_n)+\frac{\psi(\oq_n)-\psi(\uq_n)}{\oq_n-\uq_n}(q-\uq_n)$ for $q\in [\uq_n,\oq_n]$. From~\eqref{def_curlies}, this gives 
\begin{equation}\label{envconvjoue}
   \forall q\in (\uq_n,\oq_n],\;F_{\nu_{\bar{\mathcal P}(\mu)}}^{-1}(q)=F_{\nu}^{-1}(q)+\frac{\psi(\oq_n)-\psi(\uq_n)}{\oq_n-\uq_n}\mbox{ and }F_{\mu_{\underline{\mathcal P}(\nu)}}^{-1}(q)=F_{\mu}^{-1}(q)-\frac{\psi(\oq_n)-\psi(\uq_n)}{\oq_n-\uq_n}.
\end{equation} 
Any point $q$ in $(0,1)\setminus \cup_{1\le n\le N} (\uq_n,\oq_n]$ 
 is the limit of an increasing sequence $(q_k)_{k\ge 1}$ of points in $(0,1)\setminus \cup_{1\le n\le N} (\uq_n,\oq_n)$. Since, by \eqref{def_curlies} and \eqref{egalfconvf}, $\frac{1}{q-q_k}\int_{q_k}^qF_{\nu_{\bar{\mathcal P}(\mu)}}^{-1}(p)dp=\frac{1}{q-q_k}\int_{q_k}^qF_{\mu}^{-1}(p)dp$ and $\frac{1}{q-q_k}\int_{q_k}^qF_{\mu_{\underline{\mathcal P}(\nu)}}^{-1}(p)dp=\frac{1}{q-q_k}\int_{q_k}^qF_{\nu}^{-1}(p)dp$, the left-continuity of the quantile functions implies that $F_{\nu_{\bar{\mathcal P}(\mu)}}^{-1}(q)=F_{\mu}^{-1}(q)$ and $F_{\mu_{\underline{\mathcal P}(\nu)}}^{-1}(q)=F_{\nu}^{-1}(q)$. We deduce that \begin{equation}\label{equalquantqq}
\forall q\in (0,1)\setminus\cup_{1\le n\le N} (\uq_n,\oq_n],\;  F_{\nu_{\bar{\mathcal P}(\mu)}}^{-1}(q)=F_{\mu}^{-1}(q)\mbox{ and }F_{\mu_{\underline{\mathcal P}(\nu)}}^{-1}(q)=F_{\nu}^{-1}(q).
\end{equation}

By Corollary~\ref{cor_equality_wass}, there exists an optimal transport map $\tilde T$ between $\nu_{\bar{\mathcal P}(\mu)}$ and $\nu$. By Proposition~2.17 in~\cite{santambrogio}, we have $dq$-a.e. $\tilde{T}(F_{\nu_{\bar{\mathcal P}(\mu)}}^{-1}(q))=F_{\nu}^{-1}(q)$. For $x\in\R$ such that  $F_{\nu_{\bar{\mathcal P}(\mu)}}(x-)<F_{\nu_{\bar{\mathcal P}(\mu)}}(x)$, since $F_{\nu_{\bar{\mathcal P}(\mu)}}^{-1}$ is constant (equal to $x$) on $(F_{\nu_{\bar{\mathcal P}(\mu)}}(x-),F_{\nu_{\bar{\mathcal P}(\mu)}}(x)]$, we deduce that the left-continuous function $F_\nu^{-1}$ is also constant on this interval. Let now $x\in\R\cap\cup_{1\le n\le N}\{\ut_n,\ot_n\}$. By definition of the irreducible components, we have \begin{equation}
   F_{\nu_{\bar{\mathcal P}(\mu)}}(x-)\le F_\mu(x-)\le F_\mu(x)\le  F_{\nu_{\bar{\mathcal P}(\mu)}}(x).\label{inegbord}
\end{equation} If $F_{\nu_{\bar{\mathcal P}(\mu)}}(x-)<F_{\nu_{\bar{\mathcal P}(\mu)}}(x)$, $(F_{\nu_{\bar{\mathcal P}(\mu)}}^{-1},F_{\nu}^{-1})$ is constant and equal to $(x,F_{\nu}^{-1}(F_{\nu_{\bar{\mathcal P}(\mu)}}(x)))$ on the interval $(F_{\nu_{\bar{\mathcal P}(\mu)}}(x-),
F_{\nu_{\bar{\mathcal P}(\mu)}}(x)]$ and, by definition of $T$, $T(F_{\nu_{\bar{\mathcal P}(\mu)}}^{-1})$ and $F_{\nu}^{-1}$ are equal on this interval.

We are now going to prove that $dq$ a.e., $T(F^{-1}_{\nu_{\bar{\mathcal P}(\mu)}}(q))=F_{\nu}^{-1}(q)$, which, by Proposition~2.17 in~\cite{santambrogio}, ensures that $T$ is an optimal transport map between $\nu_{\bar{\mathcal P}(\mu)}$ and ${\nu}$. 

If $q\in (F_{\nu_{\bar{\mathcal P}(\mu)}}(\ut_n),F_{\nu_{\bar{\mathcal P}(\mu)}}(\ot_n-))$ then $F_{\nu_{\bar{\mathcal P}(\mu)}}^{-1}(q)\in(\ut_n,\ot_n)$ and $T(F_{\nu_{\bar{\mathcal P}(\mu)}}^{-1}(q))=F_{\nu_{\bar{\mathcal P}(\mu)}}^{-1}(q)-\frac{\psi(\oq_n)-\psi(\uq_n)}{\oq_n-\uq_n}$ with the right-hand side equal to $F_{\nu}^{-1}(q)$ by \eqref{envconvjoue} since, by \eqref{inegbord}, $$(F_{\nu_{\bar{\mathcal P}(\mu)}}(\ut_n),F_{\nu_{\bar{\mathcal P}(\mu)}}(\ot_n-))\subset (\uq_n,\oq_n).$$ By the above reasoning for $x\in\R\cap\cup_{1\le n\le N}\{\ut_n,\ot_n\}$, the equality between $T(F_{\nu_{\bar{\mathcal P}(\mu)}}^{-1})$ and $F_\nu^{-1}$ still holds on $(F_{\nu_{\bar{\mathcal P}(\mu)}}(\ut_n-),F_{\nu_{\bar{\mathcal P}(\mu)}}(\ot_n-))\cup(F_{\nu_{\bar{\mathcal P}(\mu)}}(\ot_n-),F_{\nu_{\bar{\mathcal P}(\mu)}}(\ot_n)]$.

If $q\notin (F_{\nu_{\bar{\mathcal P}(\mu)}}(\ut_n-),F_{\nu_{\bar{\mathcal P}(\mu)}}(\ot_n)]$, then $F_{\nu_{\bar{\mathcal P}(\mu)}}^{-1}(q)\le\ut_n$ or $F_{\nu_{\bar{\mathcal P}(\mu)}}^{-1}(q)>\ot_n$. We deduce that for $q\notin \cup_{1\le n\le N}(F_{\nu_{\bar{\mathcal P}(\mu)}}(\ut_n-),F_{\nu_{\bar{\mathcal P}(\mu)}}(\ot_n)]$, $F_{\nu_{\bar{\mathcal P}(\mu)}}^{-1}(q)\notin\cup_{1\le n\le N}(\ut_n,\ot_n)$ and $T(F_{\nu_{\bar{\mathcal P}(\mu)}}^{-1}(q))=F_\nu^{-1}(F_{\nu_{\bar{\mathcal P}(\mu)}}(F_{\nu_{\bar{\mathcal P}(\mu)}}^{-1}(q)))$. The right-hand side is equal to $F_\nu^{-1}(q)$ when $F_{\nu_{\bar{\mathcal P}(\mu)}}(F_{\nu_{\bar{\mathcal P}(\mu)}}^{-1}(q)-)=F_{\nu_{\bar{\mathcal P}(\mu)}}(F_{\nu_{\bar{\mathcal P}(\mu)}}^{-1}(q))$ since then $F_{\nu_{\bar{\mathcal P}(\mu)}}(F_{\nu_{\bar{\mathcal P}(\mu)}}^{-1}(q))=q$ and otherwise when $q>F_{\nu_{\bar{\mathcal P}(\mu)}}(F_{\nu_{\bar{\mathcal P}(\mu)}}^{-1}(q)-)$ since, then, the interval $(F_{\nu_{\bar{\mathcal P}(\mu)}}(F_{\nu_{\bar{\mathcal P}(\mu)}}^{-1}(q)-),F_{\nu_{\bar{\mathcal P}(\mu)}}(F_{\nu_{\bar{\mathcal P}(\mu)}}^{-1}(q))]$ on which $F_\nu^{-1}$ is constant contains $q$. 

In conclusion $T(F_{\nu_{\bar{\mathcal P}(\mu)}}^{-1}(q))=F_\nu^{-1}(q)$ for $q$ outside the at most countable set $\{F_{\nu_{\bar{\mathcal P}(\mu)}}(\ot_n-):1\le n\le N\}\cup\{F_{\nu_{\bar{\mathcal P}(\mu)}}(x-):\;x\in\R\mbox{ s.t. }F_{\nu_{\bar{\mathcal P}(\mu)}}(x-)<F_{\nu_{\bar{\mathcal P}(\mu)}}(x)\}$ and therefore $dq$ a.e..

With \eqref{equalquantqq}, we deduce that $dq$ a.e. on $(0,1)\setminus\cup_{1\le n\le N}(\uq_n,\oq_n]$, $T(F_\mu^{-1}(q))=F_{\mu_{\underline{\mathcal P}(\nu)}}^{-1}(q)$. If $q\in(\uq_n,\oq_n)$ for some $1\le n\le N$, then $F_\mu^{-1}(q)\in(\ut_n,\ot_n)$ and, by definition of $T$ and \eqref{envconvjoue}, $T(F_\mu^{-1}(q))=F_\mu^{-1}(q)-\frac{\psi(\oq_n)-\psi(\uq_n)}{\oq_n-\uq_n}=F_{\mu_{\underline{\mathcal P}(\nu)}}^{-1}(q)$. Hence $dq$ a.e. $T(F_\mu^{-1}(q))=F_{\mu_{\underline{\mathcal P}(\nu)}}^{-1}(q)$ and $T$ is an optimal transport map between $\mu$ and $\mu_{\underline{\mathcal P}(\nu)}$.
\end{proof}

\section{Numerical experiments}\label{secnum}

\subsection{Wasserstein distance}

We start by illustrating numerically the convergences obtained in Proposition~\ref{propconv2}, and deduced from Theorem~\ref{prop_curlyvee}. We present on an example the convergence of the Wasserstein projection $(\mu_I)^\rho_{\underline{\mathcal P}(\nu_I)}$ (resp. $(\nu_I)^\rho_{\bar{\mathcal P}(\mu_I)}$) toward $\mu$ (resp. $\nu$) for the Wasserstein distance when   $\mu_I$ and $\nu_I$ are the respective empirical measures of $\mu$ and $\nu$ with $\mu \lecx \nu$.  To do so we consider an example in dimension one with $\rho=2$, so that the projections can be calculated explicitly according to Theorems~\ref{propprojd1} and~\ref{prop_curlyvee}. We take $\mu=\mathcal{N}(0,1)$ and $\nu=\mathcal{N}(0,1.1)$. For $I\ge 1$, we consider independent samples $X_1,\dots,X_I$ and $Y_1,\dots,Y_I$ distributed respectively according to $\mu$ and $\nu$. Then, we set $\mu_I=\frac{1}{I} \sum_{i=1}^I \delta_{X_i}$, $\nu_I=\frac{1}{I} \sum_{i=1}^I \delta_{Y_i}$, $\bar{X}_I=\frac{1}{I}\sum_{i=1}^IX_i$, $\bar{Y}_I=\frac{1}{I}\sum_{i=1}^IY_i$, $\tmu_I=\frac{1}{I} \sum_{i=1}^I \delta_{X_i-\bar{X}_I}$ and $\tnu_I=\frac{1}{I} \sum_{i=1}^I \delta_{Y_i-\bar{Y}_I}$. Notice that, to define $\tmu_I$ and $\tnu_I$, we took advantage of the knowledge of the common mean of $\mu$ and $\nu$. This situation is usual in financial applications : discounted asset prices are martingales and their means are given by the present values. We calculate the Wasserstein projections $(\mu_I)_{\underline{\mathcal P}(\nu_I)}$ and $(\tmu_I)_{\underline{\mathcal P}(\tnu_I)}$ (resp. $(\nu_I)_{\bar{\mathcal P}(\mu_I)}$ and $(\tnu_I)_{\bar{\mathcal P}(\tmu_I)}$) and the $2$-Wasserstein distance between each of these measures and $\mu$ (resp. $\nu$), as explained below. 

As a comparison to these projections, we consider the respective approximations of $\mu$ and $\nu$ by $\mu_I\wedge \nu_I$ and $\mu_I\vee\nu_I$, where $\mu_I\wedge \nu_I$ and $\mu_I\vee \nu_I$ are respectively defined as the infimum and the supremum of $\mu_I$ and $\nu_I$ for the decreasing convex order when $\frac{1}{I}\sum_{i=1}^I X_i\le \frac{1}{I}\sum_{i=1}^I Y_i$ and for the increasing convex order otherwise so that $\mu_I\wedge \nu_I\in\underline{\mathcal P}(\nu_I)$ and $\mu_I\vee\nu_I\in\bar{\mathcal P}(\mu_I)$. We also consider the approximations by $\tmu_I\wedge \tnu_I$ and $\tmu_I\vee\tnu_I$. These approximations can be calculated explicitly for probability measures with finite support (see~\cite{ACJ3} or~\cite{ACJ1}) and are natural alternatives to the Wasserstein projections in dimension~1.

\begin{figure}[h]
  \centering
  \includegraphics[width=\linewidth]{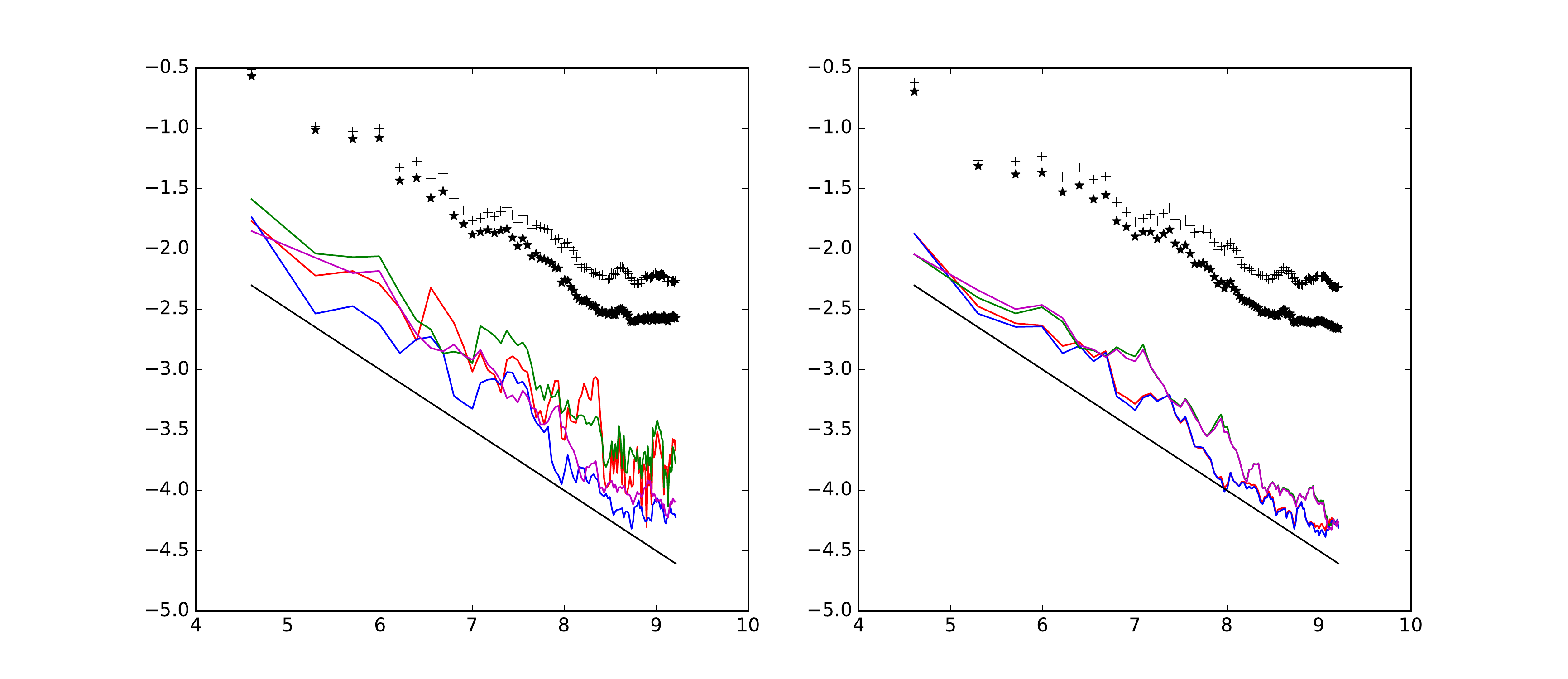}
  \caption{Plot of the logarithms of the Wasserstein distances in function of~$\log(I)$. } \label{Fig_Wass_cv_1}
\end{figure}
The graph at left (resp. right) of Figure~\ref{Fig_Wass_cv_1} illustrates the convergence of $W_2(\mu,\mu_I\wedge \nu_I)$, $W_2(\mu,(\mu_I)_{\underline{\mathcal P}(\nu_I)})$, $W_2(\mu_I\vee \nu_I,\nu)$ and $W_2((\nu_I)_{\bar{\mathcal P}(\mu_I)},\nu)$ (resp. $W_2(\mu,\tmu_I\wedge \tnu_I)$, $W_2(\mu,(\tmu_I)_{\underline{\mathcal P}(\tnu_I)})$, $W_2(\tmu_I\vee \tnu_I,\nu)$ and $W_2((\tnu_I)_{\bar{\mathcal P}(\tmu_I)},\nu)$ ) toward zero as $I\to\infty$. The corresponding curves are respectively in red, blue, green and magenta. The star (resp. cross) points indicate the upper bound for $W_2(\mu,(\mu_I)_{\underline{\mathcal P}(\nu_I)})$  (left) and $W_2(\mu,(\tmu_I)_{\underline{\mathcal P}(\tnu_I)})$ (right) (resp. $W_2((\nu_I)_{\bar{\mathcal P}(\mu_I)},\nu)$ (left)  and $W_2((\tnu_I)_{\bar{\mathcal P}(\tmu_I)},\nu)$ (right)) given by Proposition~\ref{prop_cv_curlywedge} (resp.Theorem~\ref{prop_curlyvee}). 
As expected, the curves in blue and magenta are below these points. Let us mention that all these Wasserstein distances are calculated exactly by using the quantile function ${\mathcal N}^{-1}$ of the standard normal variable. For instance, if $\eta=\sum_{i=1}^Ip_i\delta_{Z_i}$ with $Z_1\le Z_2\le\hdots\le Z_I$, $P_0=0$ and $P_i=P_{i-1}+p_i$ for $1\le i\le I$, 
\begin{align*}
   W_2^2(\mu, \eta)&=\int_\R x^2(\mu(dx)+\eta(dx))-2 \sum_{i=1}^I Z_i\int_{P_{i-1}}^{P_i}{\mathcal N}^{-1}(p)dp\\&=1+\sum_{i=1}^Ip_iZ_i^2+\frac{\sqrt{2}}{\sqrt{\pi}}\sum_{i=1}^IZ_i\left(e^{-({\mathcal N}^{-1}(P_i))^2/2}-e^{-({\mathcal N}^{-1}(P_{i-1}))^2/2}\right).
\end{align*}
Asymptotically, the measure $(\mu_I)_{\underline{\mathcal P}(\nu_I)}$ (resp.  $(\nu_I)_{\bar{\mathcal P}(\mu_I)}$) seems to slightly better approximate $\mu$ (resp. $\nu$) than  $\mu_I\wedge \nu_I$ (resp. $\mu_I\vee \nu_I$). Nonetheless, all these measures seem to converge for the Wasserstein distance at a rate close to $O(I^{-1/2})$ as indicated by the line in black with equation  $y=-x/2$. This rate is better than the theoretical one stated in Proposition~\ref{propconv2}. In the right figure, we first observe that equalizing the means improves the approximations and reduces the Wasserstein distances (see the distances to the black lines). However, the rate of convergence is still roughly in $O(I^{-1/2})$. We also observe that there are only very small differences between using $\tmu_I\wedge \tnu_I$ or $(\tmu_I)_{\underline{\mathcal P}(\tnu_I)}$ (resp.  $\tmu_I\vee \tnu_I$ or $(\tnu_I)_{\bar{\mathcal P}(\tmu_I)}$).
\begin{figure}[h]
  \centering
  \includegraphics[width=\linewidth]{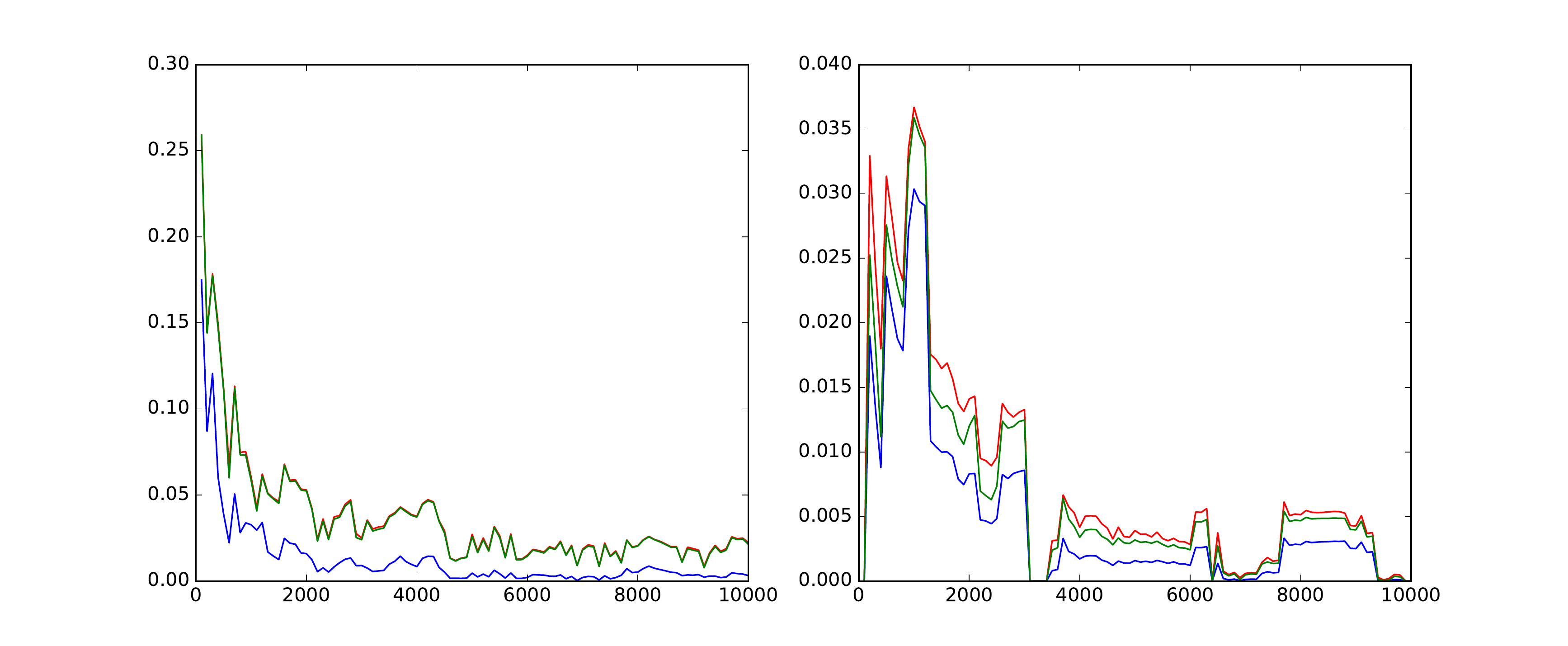}
  \caption{Plot of the Wasserstein distances $W_2(\mu_I,\mu_I\wedge\nu_I)$, $W_2(\mu_I,(\mu_I)_{\underline{\mathcal P}(\nu_I)})$, $W_2(\nu_I,\mu_I\vee\nu_I)$ (left) and  $W_2(\tmu_I,\tmu_I\wedge\tnu_I)$, $W_2(\tmu_I,(\tmu_I)_{\underline{\mathcal P}(\tnu_I)})$, $W_2(\tnu_I,\tmu_I\vee\tnu_I)$ (right) in function of~$I$. } \label{Fig_Wass_cv_2}
\end{figure}

In Figure~\ref{Fig_Wass_cv_2} are plotted at left (resp. right) the values of $W_2(\mu_I,\mu_I\wedge\nu_I)$, $W_2(\mu_I,(\mu_I)_{\underline{\mathcal P}(\nu_I)})=W_2(\nu_I,(\nu_I)_{\bar{\mathcal P}(\mu_I)}) $,  $W_2(\nu_I,\mu_I\vee\nu_I)$ (resp. $W_2(\tmu_I,\tmu_I\wedge\tnu_I)$, $W_2(\tmu_I,(\tmu_I)_{\underline{\mathcal P}(\tnu_I)})=W_2(\tnu_I,(\tnu_I)_{\bar{\mathcal P}(\tmu_I)}) $,  $W_2(\tnu_I,\tmu_I\vee\tnu_I)$) in function of~$I$. The corresponding curves are in red, blue and green. We observe that the values of  $W_2(\mu_I,\mu_I\wedge\nu_I)$ and  $W_2(\nu_I,\mu_I\vee\nu_I)$ are very close. As expected, the blue curve is below the two other ones. At right, we observe that all the Wasserstein distances are equal to $0$ on our sample for $I \approx 3200$, but take again positive values for larger values of~$I$. This shows that the value of $I$ from which we have $\tmu_I \lecx \tnu_I$, if it exists, depends on the sample and may be large.

Now, we conclude this section by checking the accuracy of the solver COIN-OR\footnote{{\tt https://www.coin-or.org/}} for the quadratic optimization problem~\eqref{miniquad} with $\rho=2$. In fact, in dimension~1, we know that $(\mu_I)_{\underline{\mathcal P}(\nu_I)}$ can be calculated explicitly as described below Theorem~\ref{propprojd1}. In Table~\ref{table_COINOR}, we calculate the Wasserstein distance between $(\mu_I)_{\underline{\mathcal P}(\nu_I)}$ and the measure  obtained by solving numerically~\eqref{miniquad} with COIN-OR for different sample sizes $I$. 
\begin{table}[H]
  \begin{centering}
    \begin{tabular}{|r||c|c|c|c|c|}
      \hline
      $I$ & 10 & 50 & 100 & 200 & 300\\
      \hline
      $W_2$-Wasserstein distance & $4.4 \times 10^{-5}$  & $1.4\times 10^{-6}$ & $4.5 \times 10^{-6}$ & $4.1\times 10^{-7}$ & $4.2\times 10^{-7}$ \\
      \hline
    \end{tabular}
  \end{centering}
  \caption{Comparison of the numerical minimizer of~\eqref{miniquad} for $\rho=2$ with the explicit solution $(\mu_I)_{\underline{\mathcal P}(\nu_I)}$. }\label{table_COINOR}
\end{table}
As expected, the difference is very small.  This validates numerically our theoretical results. More importantly, this indicates that the solver is reliable for finding the optimal solution with the values of~$I$ that we have considered in this paper.

\subsection{ MOT problems in dimension~2 with two marginal laws.}

\subsubsection{An explicit example}

Let $\mu$ and $\nu$ be respectively the uniform distributions on $[-1,1]^2$ and $[-2,2]^2$. For $x=(x^1,x^2) \in \R^2$ and $y=(y^1,y^2) \in \R^2$, we consider the minimization of the cost function $c(x,y)=|x^1-y^1|^\rho+|x^2-y^2|^\rho$, with $\rho>2$. For any $\pi \in \Pi^M(\mu,\nu)$, we have $\int_{\R^2 \times \R^2} \|y-x\|_2^2 \pi(dx,dy)=\int_{\R^2 } \|y\|_2^2 \nu(dy) - \int_{\R^2 } \|x\|_2^2 \mu(dx)=2$. Jensen's inequality gives
\begin{align*}
  \int_{\R^2 \times \R^2} |x^1-y^1|^\rho&+|x^2-y^2|^\rho  \pi(dx,dy)\\& \ge \left( \int_{\R^2 \times \R^2} |x^1-y^1|^2  \pi(dx,dy) \right)^{\frac \rho 2} + \left(\int_{\R^2 \times \R^2} |x^2-y^2|^2  \pi(dx,dy) \right)^{\frac \rho 2}=2. 
\end{align*}
The equality condition in Jensen's equality gives that $|x^1-y^1|=|x^2-y^2|=1$, $\pi(dx,dy)$-almost surely. Now, let us consider $X=(X^1,X^2)$ be distributed according to $\mu$ and $Z=(Z^1,Z^2)$ a couple of independent Rademacher random variables which is independent of $X$. Then $Y=X+Z$ is distributed according to $\nu$ and satisfies $|Y^1-X^1|=|Y^2-X^2|=1$. The probability distribution $\pi^\star$ of $(X,Y)$ is the unique martingale optimal coupling that minimizes $\int_{\R^2 \times \R^2} c(x,y) \pi(dx,dy)$. Indeed, if $(\tilde{X},\tilde{Y})$ is distributed according to an optimal coupling, then $\tilde{Y}^1-\tilde{X}^1$ and $\tilde{Y}^2-\tilde{X}^2$ follow the Rademacher distribution, and both these random variables are necessarily independent of $\tilde{X}$ in order to satisfy the martingale property. Last,  $\tilde{Y}^1-\tilde{X}^1$ and $\tilde{Y}^2-\tilde{X}^2$ are necessarily independent, otherwise $\tilde{Y}$ would not follow~$\nu$.

We now illustrate the MOT and consider independent samples $(X^1_1,X^2_1),\dots,(X^1_I,X^2_I)$ and $(Y^1_1,Y^2_1),\dots,(Y^1_I,Y^2_I)$ respectively distributed according to $\mu$ and $\nu$. We set $\tmu_I= \frac 1I \sum_{i=1}^I \delta_{(X^1_i -\bar{X}_I^1,X^2_i -\bar{X}_I^2)}$ and $\tnu_I=\frac 1I \sum_{i=1}^I \delta_{(Y^1_i-\bar{Y}_I^1,Y^2_i-\bar{Y}_I^2)}$, with $\bar{X}^\ell_I=\frac 1I \sum_{i=1}^I X^\ell_i$ and $\bar{Y}^\ell_I=\frac 1I \sum_{i=1}^I Y^\ell_i$. We work with $\tmu_I$ and $\tnu_I$ rather than with the empirical measures $\mu_I$ and $\nu_I$ since we have noticed on our experiments that they better approximate $\mu$ and $\nu$ (see Figure~\ref{Fig_Wass_cv_1}) and give better results for the approximation of MOT problems (see~\cite{ACJ3}). Let us mention here that in financial applications, it is generally  possible to calculate $\tmu_I$ and $\tnu_I$ from the empirical measures $\mu_I$ and $\nu_J$ since the mean of $\mu$ and $\nu$ is given by the current price of the underlying assets.
To calculate $(\tmu_I)^2_{\underline{\mathcal P}(\tnu_I)}$, we have to solve the quadratic optimization problem with linear constraints described in equation~\eqref{miniquad} for $\rho=2$. The dimension of the problem is thus equal to~$I^2$. We have used the COIN-OR solver in our numerical experiments, which enables us to solve~\eqref{miniquad} for $I$ up to $500$. Once $(\tmu_I)^2_{\underline{\mathcal P}(\tnu_I)}=\frac 1I \sum_{i=1}^I \delta_{(\tilde{X}^1_i,\tilde{X}^2_i)^\star}$ is calculated, we can then solve the discrete MOT problem between $(\tmu_I)^2_{\underline{\mathcal P}(\tnu_I)}$ and $\tnu_I$. 
\begin{figure}[h]
\begin{centering}
  \includegraphics[width=\textwidth]{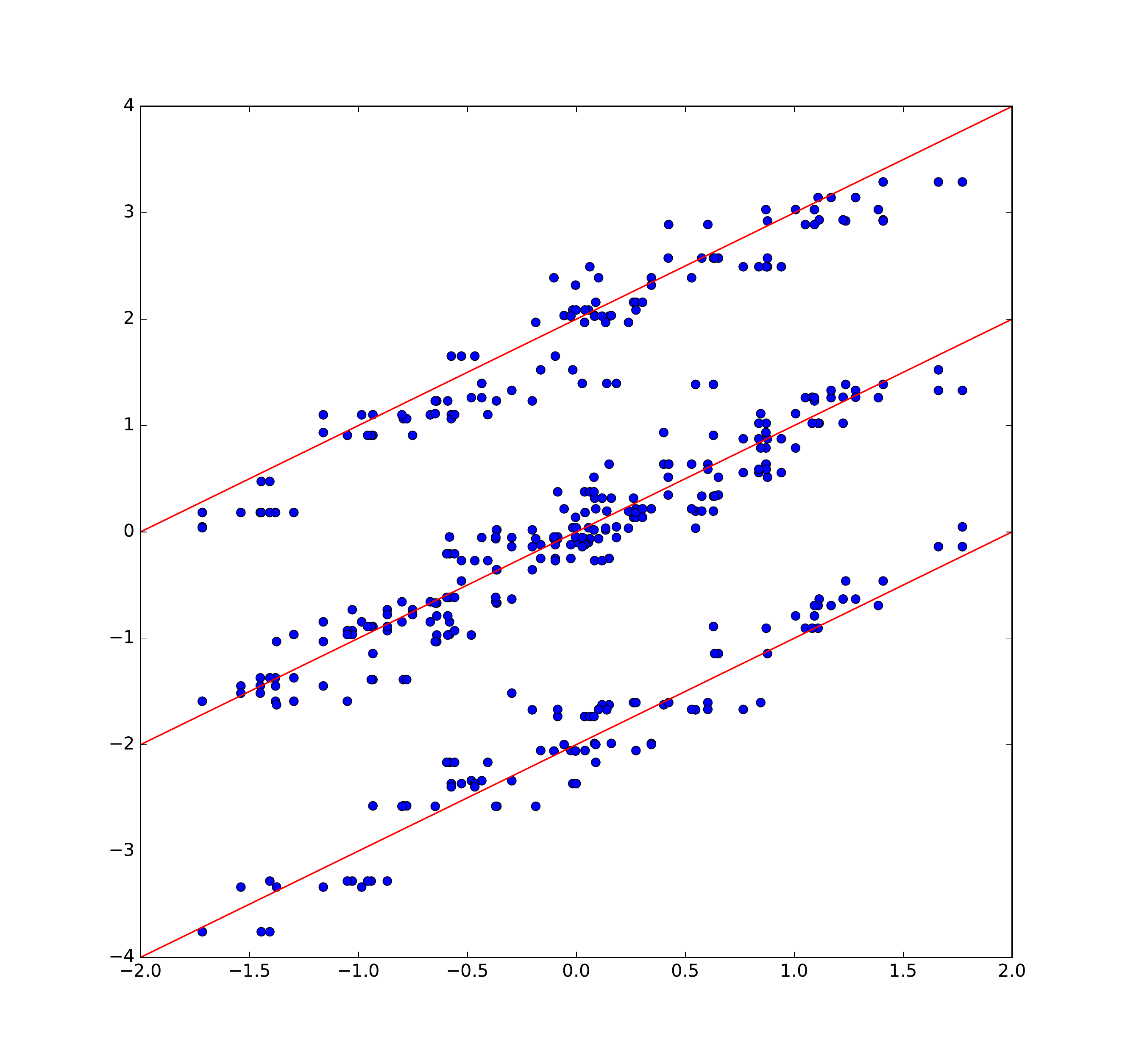}
  \caption{Plot of  $y^2_i-y^1_i$ in function of $x^2_i-x^1_i$ for the points $(x_i,y_i)$ with positive probability in the MOT for  $((\tmu_I)^2_{\underline{\mathcal P}(\tnu_I)},\nu_I)$, with $I=100$. In red are drawn the lines $y=x-2$, $y=x$ and $y=x+2$.} \label{2d_parallel}
\end{centering}
\end{figure}

In Figure~\ref{2d_parallel}, for $\rho=2.5$, we have plotted $y^2_i-x^2_i$ in function of $y^1_i-x^1_i$ for the points $(x_i,y_i)$ with positive probability in the MOT for  $((\tmu_I)^2_{\underline{\mathcal P}(\tnu_I)},\nu_I)$. We recall that the optimal coupling for the continuous MOT is given by $(X,Y)$ with $X\sim \mu$ and $Y=X+Z$, $Z$ being a couple of independent Rademacher random variables. Since $Y_2-Y_1=X_2-X_1+Z_2-Z_1$ and $Z_2-Z_1$ takes values in $\{-2,0,2\}$, we expect to observe that the points are gathered around the lines $y=x-2$, $y=x$ and $y=x+2$, which is the case on Figure~\ref{2d_parallel}.  This checks our implementation of the algorithm. Besides, we have calculated on $100$ independent runs the value of the discrete MOT for  $((\tmu_I)^2_{\underline{\mathcal P}(\tnu_I)},\nu_I)$ with $I=100$: the average is equal to $2.0064$ and the standard deviation is equal to $0.2213$, which gives $[1.9631,2.0498]$ as $95\%$ confidence interval, which approximates well the value of the continuous MOT.

\subsubsection{Model-free bounds on a best-of option}

Let $(G^1,G^2)$ be a centered Gaussian vector with covariance matrix $\Sigma$. We denote by $\mu$ the law of $(X^1,X^2)$ with $X^\ell=\exp(G^\ell-\Sigma_{\ell \ell}/2)$ for $\ell\in \{1,2\}$, and  by $\nu$ the law of  $(Y^1,Y^2)$ with $Y^\ell=\exp(\sqrt{2}G^\ell-\Sigma_{\ell \ell})$. In the financial context, this choice of marginal laws is usual and corresponds to a two-dimensional Black-Scholes model: $(X^1,X^2)$ is the price of two assets at time~$t>0$ and $(Y^1,Y^2)$ is the price of these assets at time~$2t$. We are interested in an option that pays $\max(Y^1-X^1,Y^2-X^2,0)$, i.e. the best arithmetic performance of the two assets, if positive. The price of this option in the Black-Scholes model can easily be calculated by using a Monte-Carlo algorithm.

Let $(X^1_1,X^2_1),\dots,(X^1_I,X^2_I)$ and $(Y^1_1,Y^2_1),\dots,(Y^1_I,Y^2_I)$ denote independent samples respectively distributed according to $\mu$ and $\nu$. We set $\tmu_I= \frac 1I \sum_{i=1}^I \delta_{(\tilde{X}^1_i,\tilde{X}^2_i)}$ and $\tnu_I=\frac 1I \sum_{i=1}^I \delta_{(\tilde{Y}^1_i,\tilde{Y}^2_i)}$, with $(\tilde{X}^1_i,\tilde{X}^2_i)=(X^1_i +1-\bar{X}_I^1,X^2_i +1-\bar{X}_I^2)$,  $(\tilde{Y}^1_i,\tilde{Y}^2_i)=(Y^1_i +1-\bar{Y}_I^1,Y^2_i +1-\bar{Y}_I^2)$, $\bar{X}^\ell_I=\frac 1I \sum_{i=1}^I X^\ell_i$ and $\bar{Y}^\ell_I=\frac 1I \sum_{i=1}^I Y^\ell_i$. 
We calculate $(\tmu_I)^2_{\underline{\mathcal P}(\tnu_I)}$ numerically by using again the quadratic optimization solver COIN-OR, and then solve the discrete MOT problem between $(\tmu_I)^2_{\underline{\mathcal P}(\tnu_I)}$ and $\tnu_I$. 

\begin{figure}
\begin{centering}
  \includegraphics[width=\textwidth]{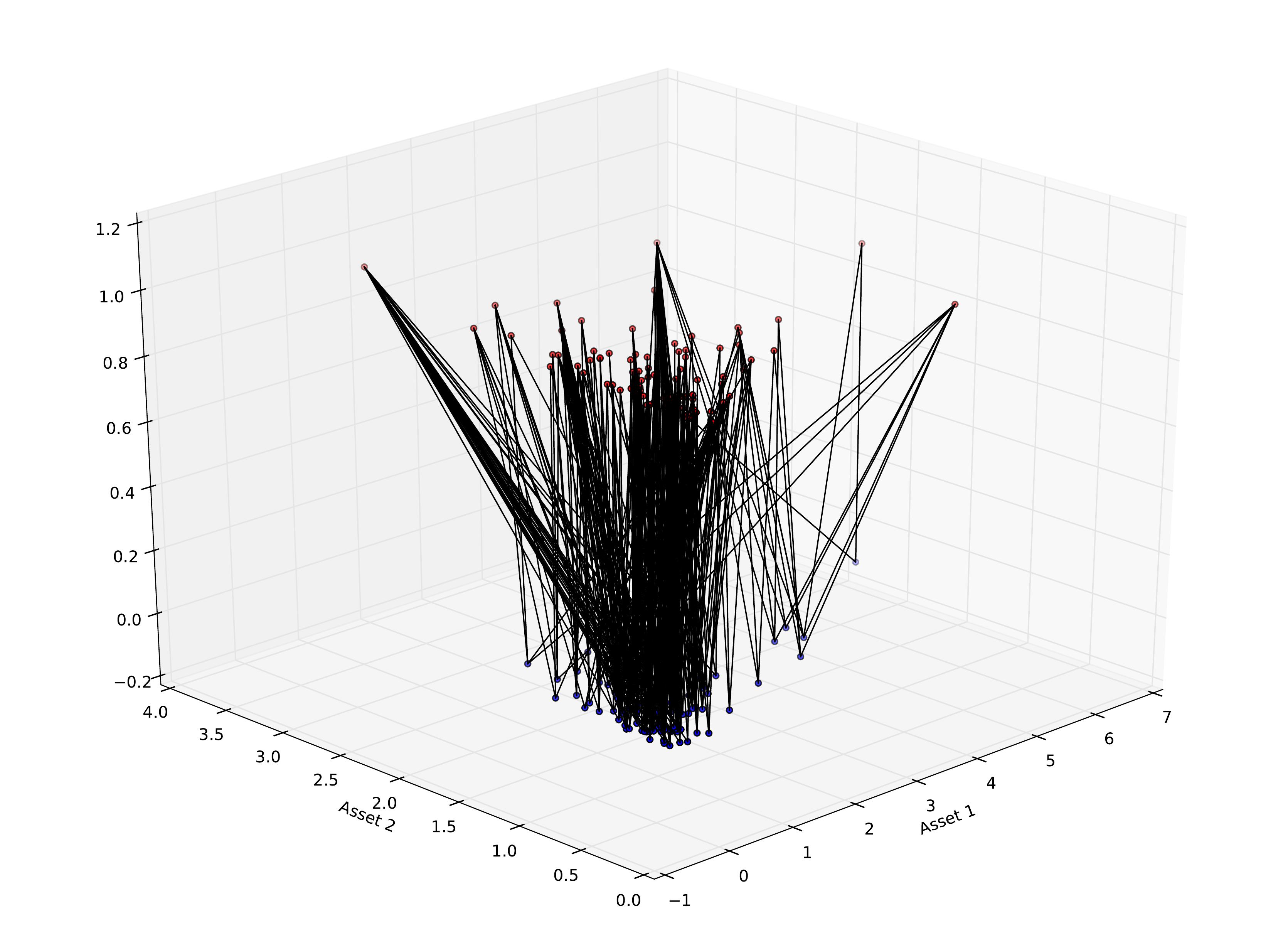}

  \includegraphics[width=\textwidth]{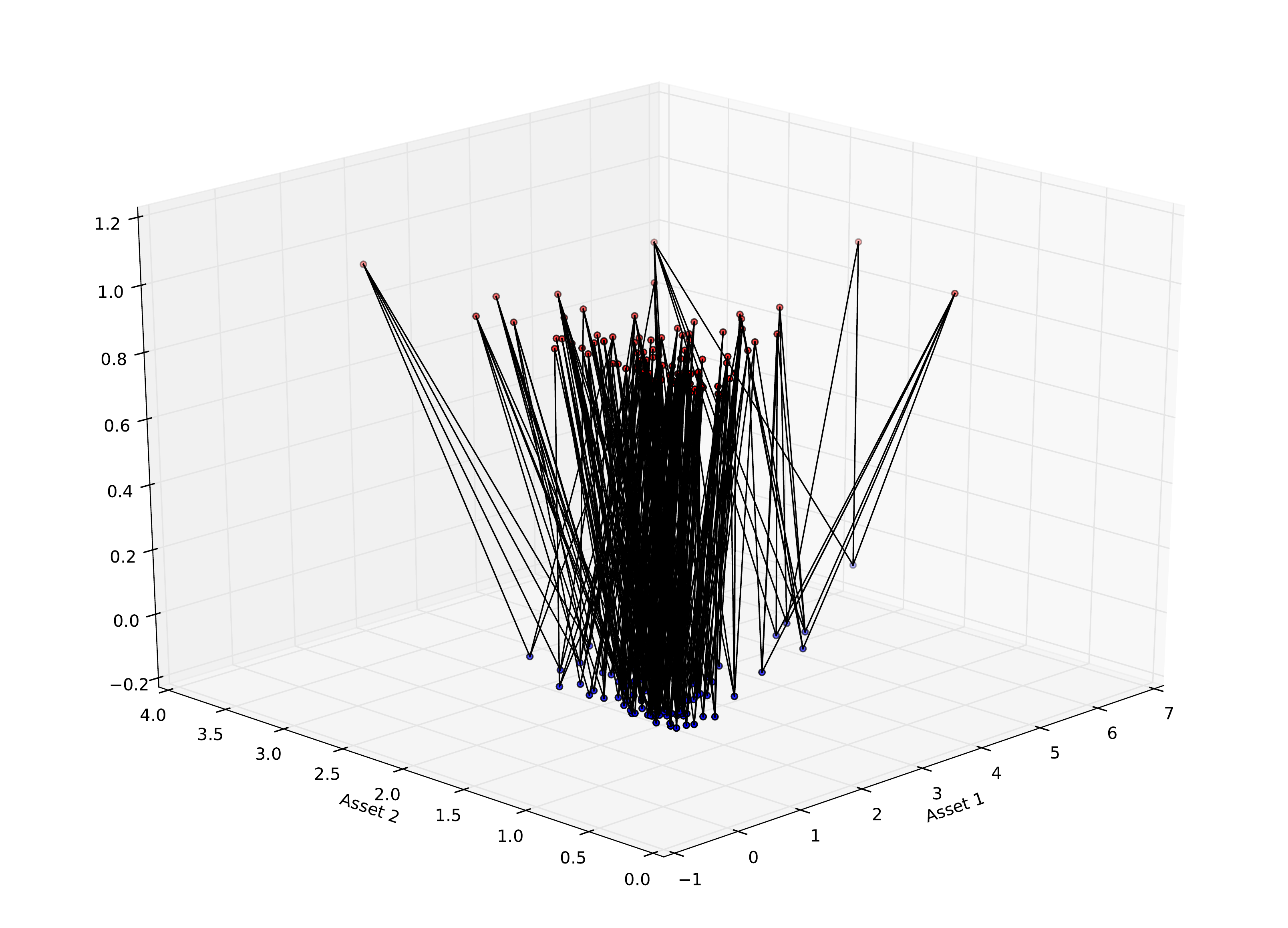}
\end{centering}
 \caption{Discrete MOT ($I=100$) in dimension~2 for the minimization problem (top) and the maximization problem (bottom). } \label{LN_2d}
\end{figure}
We now turn to our example illustrated in Figure~\ref{LN_2d}. We have considered the following covariance matrix $\Sigma=\left[ \begin{matrix} 0.5 & 0.1 \\ 0.1 & 0.1 \end{matrix} \right]$. With this choice, the Black-Scholes price of the option is approximately equal to $0.345$. With $I=100$, we have calculated on $100$ independent runs  the value of the minimization and the maximization programs, and then computed the mean values. We have thus obtained $0.2293$ for the lower bound price and $0.4111$ for the upper bound price. The corresponding standard variations are respectively $0.0848$ and $0.1422$, which makes 95\% confidence intervals with half lengths $0.017$ and $0.028$. In Figure~\ref{LN_2d}, we have plotted the discrete MOT on the same sample for the minimization and the maximization problem. Precisely, we have plotted  the points $(\tilde{X}^1_i,\tilde{X}^2_i)^\star$, $i\in\{1,\dots,I\}$ in the hyperplane $z=0$ and the points $(\tilde{Y}^1_i,\tilde{Y}^2_i)$ in the hyperplane $z=1$. The edges between the points $(\tilde{X}^1_i,\tilde{X}^2_i)^\star$ and $(\tilde{Y}^1_j,\tilde{Y}^2_j)$ indicate that the optimal coupling gives a positive weight to the corresponding transitions. The difference between the two optimal couplings is clear. We can heuristically explain the graphs as follows. The cost function $c(x,y)=\max(y^1-x^1,y^2-x^2,0)$ will anyway be positive for a large increase of one of the two assets. Therefore, to minimize the cost, one has to gather the large increases of Asset~$1$ and Asset~$2$. Instead, to maximize the cost, it is better to gather an increase of one asset with a decrease of the other one.

The CPU time needed for the computation of the Wasserstein projection and for the linear programming problem is reported in Table~\ref{table_COINOR2}. The dimension $d=1$ rows of the table correspond to the MOT problem between the laws of $X_1$ and $Y_1$ for the cost function $\max(y-x,0)$. What mainly influences the computation time is the dimension $I^2$ in which the optimal matrix $(r_{ij})$ has to be found. The dimension $d$ of the underlying space of the probability measures has a low impact on the computation time for the quadratic problem~\eqref{miniquad}, since the number of equality constraints $2I$ does not change with $d$. Instead, it has some impact on the linear programming problem~\eqref{MOT_discrete}, since the number of equality constraints $(2+d)I$ increases with $d$. Nonetheless, since the resolution of the linear problem is much less time consuming  than the resolution of the quadratic problem, the impact of the dimension $d$ on the overall computation time is rather mild. 
\begin{table}[H]
  \begin{centering}
    \begin{tabular}{|r||c|c|c|c|c|}
      \hline
      $I$ & 100 & 150 & 200 & 300 & 500\\
      \hline
      Quadratic problem~\eqref{miniquad}, $d=1$ & $1.5 s$  & $4.8s$ & $18s$ & $88s$ & $673s$ \\
      Quadratic problem~\eqref{miniquad}, $d=2$ & $1.3 s$  & $10s$ & $22 s$ & $105s$ & $807s$ \\
      \hline      
      Linear problem~\eqref{MOT_discrete}, $d=1$  & $0.3s$  & $0.78s$ & $2s$ & $6.6s$ & $41s$ \\
      Linear problem~\eqref{MOT_discrete}, $d=2$  & $0.43s$  & $2s$ & $4.5s$ & $19.5s$ & $120s$ \\
      \hline
    \end{tabular}
  \end{centering}
  \caption{Computation time on a CPU Intel Core i7 at 2.6GHz with COIN-OR of the quadratic and linear problems in dimensions $d=1$ and $d=2$. }\label{table_COINOR2}
\end{table}

\subsection{Further directions}

In view of Propositions~\ref{prop_cv_curlywedge} and~\ref{propconv2}, it would be nice to prove the stability of $$\inf_{\pi\in\Pi^M(\mu,\nu)}\int_{\R^d\times\R^d}c(x,y)\pi(dx,dy)$$ with respect to $\mu$ and $\nu$ in ${\mathcal P}(\R^d)$ for the weak convergence topology or the Wasserstein distance. On our numerical example of Figure~\ref{2d_parallel} where the continuous MOT is explicit, the convergence of the discrete optimal cost towards the continuous one seems to hold.
We plan to investigate this property in a future work. Note that for cost functions satisfying the so-called Spence-Mirrlees condition (see~\cite{HLTo}), the stability of left-curtain couplings obtained by Juillet~\cite{Juillet} is an important step in that direction. 

To overcome the sample size limitation for the linear programming solvers to compute the solution of problem  \eqref{MOT_discrete}, one can contemplate introducing an entropic regularization of this problem similar to the one proposed by Benamou et al.~\cite{BCCNP} for discrete optimal transport.
For $\mu_I=\sum_{i=1}^I p_i\delta_{x_i}\lecx\nu_J=\sum_{j=1}^J q_j\delta_{y_j}$ and $\varepsilon>0$, the regularized problem is the minimization of \begin{equation*}\sum_{i=1}^I\sum_{j=1}^J r^\varepsilon_{ij} \left(c(x_i,y_j)+\varepsilon(\ln r^\varepsilon_{ij}-1)\right)
 \end{equation*}
 under the constraints $r^\varepsilon_{ij}\ge 0,\ \sum_{i=1}^I r^\varepsilon_{ij}=q_j\mbox{ for }j\in\{1,\hdots,J\}, \ \sum_{j=1}^J r^\varepsilon_{ij}=p_i$ and $\sum_{j=1}^J r^\varepsilon_{ij}y_j=p_ix_i$ for $i\in\{1,\hdots,I\}$. Since the constraints are affine, this problem can be solved by the iterative Bregman projections presented in~\cite{BCCNP}. In particular the solution is obtained by iterating successive entropic projections on the first marginal law constraints, on the second marginal law constraints and on the martingale constraints. The two first projections are explicit (see for instance Proposition~1~\cite{BCCNP}). The entropic projection on the martingale constraints can be computed using the generalized iterative scaling algorithm introduced by Darroch and Ratcliff~\cite{DR}. Such an approach combined with a relaxation of the martingale constraint has been recently investigated by Guo and Obl\`{o}j~\cite{GuOb}.
   
\appendix 
\section{Technical lemmas}
\begin{lemma}\label{lem_ordre_cvx} Let $\mu,\nu \in \mathcal{P}_1(\R^d)$. Then, we have $\mu \lecx \nu$ if, and only if,
  $$\forall \phi:\R^d\rightarrow \R \text{ convex and such that } \sup_{x\in\R^d}\frac{|\phi(x)|}{1+|x|}<\infty, \int_{\R^d}\phi(x)\mu(dx)\le \int_{\R^d}\phi(x)\nu(dx).$$
\end{lemma}
\begin{proof}
  Let $\phi:\R^d\rightarrow \R$ be a convex function. We define $\phi^*(y)=\sup_{x\in \R^d} x \cdot y -\phi(x)$ the Legendre-Fenchel transform of $\phi$ and have
  $$\phi(x)=\phi^{**}(x)=\sup_{y\in \R^d} x \cdot y -\phi^*(y).$$
  The function $\phi^*:\R^d \rightarrow [-\phi(0),+\infty]$ is a convex lower semicontinuous function. Therefore, for any $n\ge 1$, there exists $y_n$ with Euclidean norm $|y_n|\le n$ and $\inf_{|y|\le n } \phi^*(y)=\phi^*(y_n)$ . There exists $n_0 \in \N^*$ such that $\phi^*(y_n)<\infty$ for $n\ge n_0$, otherwise we would have $\phi^*=+\infty$ and then $\phi=-\infty$.  We set $\phi_n(x)=\sup_{|y|\le n} x \cdot y -\phi^*(y)$ and have for $n\ge n_0$
  $$ x \cdot y_n- \phi^*(y_n) \le \phi_n(x) \le n |x| +\phi(0).$$
  Thus, $\phi_n$ is with affine growth and therefore $\int_{\R^d} \phi_n(x) \mu(dx) \le \int_{\R^d} \phi_n(x) \nu(dx)$.  
  By the monotone convergence theorem the integrals  $\int_{\R^d}(\phi_n-\phi_{n_0})(x)\mu(dx)$ (resp.  $\int_{\R^d}(\phi_n-\phi_{n_0})(x)\nu(dx)$) converge to  $\int_{\R^d}(\phi-\phi_{n_0})(x)\mu(dx)$ (resp.  $\int_{\R^d}(\phi-\phi_{n_0})(x)\nu(dx)$) as $n\to\infty$. We conclude that  $\int_{\R^d}\phi(x)\mu(dx)\le \int_{\R^d}\phi(x)\nu(dx)$. 
\end{proof}
\begin{lemma}\label{lemconv}
   Let $f,g:[0,1]\to\R$ be two convex functions and $h$ denote the convex hull of $f-g$. Then $f-h$ is convex.
\end{lemma}
\begin{proof}
   Let $0\le p<q\le 1$ and $\alpha\in[0,1]$. If $h(\alpha p+(1-\alpha)q)=(f-g)(\alpha p+(1-\alpha)q)$, then, using the convexity of $g$, then the fact that $h$ is bounded from above by $f-g$ for the two inequalities, we obtain that
\begin{align}
   (f-h)(\alpha p+(1-\alpha)q)&=g(\alpha p+(1-\alpha)q)\le \alpha g(p)+(1-\alpha)g(q)\notag \\&=\alpha (f(p)-(f-g)(p))+(1-\alpha)(f(q)-(f-g)(q))\notag \\&\le \alpha (f-h)(p)+(1-\alpha)(f-h)(q).\label{caseq}
\end{align}
Otherwise, $h$ is affine on some 
interval $[r,s]$ with $0\le r<\alpha p+(1-\alpha)q<s\le 1$, $h(r)=(f-g)(r)$ and $h(s)=(f-g)(s)$. If $r\in(p,\alpha p+(1-\alpha)q)$, then replacing $\alpha$ by $\frac{q-r}{q-p}$ in \eqref{caseq}, we get
$(f-h)(r)\le \frac{q-r}{q-p}(f-h)(p)+\frac{r-p}{q-p}(f-h)(q)$ so that $(f-h)(r\vee p)\le \frac{q-r\vee p}{q-p}(f-h)(p)+\frac{r\vee p-p}{q-p}(f-h)(q)$. In a symmetric way, $(f-h)(s\wedge q)\le \frac{q-s\wedge q}{q-p}(f-h)(p)+\frac{s\wedge q-p}{q-p}(f-h)(q)$. Hence, 
\begin{align*}
\frac{s\wedge q-(\alpha p+(1-\alpha)q)}{s\wedge q-r\vee p}(f-h)(r\vee p)&+\frac{(\alpha p+(1-\alpha)q)-r\vee p}{s\wedge q-r\vee p}(f-h)(s\wedge q)\\&\le \alpha (f-h)(p)+(1-\alpha)(f-h)(q).
\end{align*}
By convexity of $f$ and the affine property of $h$ on the interval $[r\vee p,s\wedge q]$ containing $\alpha p+(1-\alpha)q$, the left-hand side is not smaller than $(f-h)(\alpha p+(1-\alpha)q)$.
\end{proof}
\begin{lemma}\label{lemquant}
 Let $f:(0,1)\to\R$ be a non-decreasing function and $\eta$ denote the probability distribution of $f(U)$ for $U$ uniformly distributed on $(0,1)$. Then $f$ and the quantile function $F_\eta^{-1}$ coincide away from the at most countable set of their common discontinuities and even everywhere on $(0,1)$ if $f$ is moreover left-continuous.
\end{lemma}
\begin{proof}
The random variables $f(U)$ and $F_\eta^{-1}(U)$ are both distributed according to $\eta$. Hence for $p\in(0,1)$, 
$\P(f(U)\le F_\eta^{-1}(p))=\P(F_\eta^{-1}(U)\le F_\eta^{-1}(p))\ge p$ so that $F_\eta^{-1}(p)\ge \sup_{q\in(0,p)}f(q)$. By symmetry, $f(p)\ge\sup_{q\in(0,p)}F_\eta^{-1}(q)$ with the supremum equal to $F_\eta^{-1}(p)$ by left-continuity and monotonicity of $F_\eta^{-1}$. Hence $f(p)\ge F_\eta^{-1}(p)\ge\sup_{q\in(0,p)}f(q)$ with the supremum equal to $f(p)$ when $f$ is left-continuous. 
\end{proof}\begin{lemma}\label{leminf}
For $x\in\R$, any non empty subset ${\mathcal P}^x$ of $\{\eta\in{\mathcal P}_1(\R):\int_{\R}y\eta(dy)=x\}$ has an infimum $\pi$ for the convex order. Moreover for all $q\in[0,1]$, $\int_q^1F_\pi^{-1}(p)dp=\inf_{\eta\in{\mathcal P}^x}\int_q^1F_\eta^{-1}(p)dp$.
\end{lemma}
\begin{proof}The existence of the infimum is given by Kertz and R\"osler~\cite{KeRo2} p162.
   These authors work with the characterization of the convex order in terms of the cumulative distribution functions. By the more convenient characterization in terms of the quantile functions recalled in \eqref{caraccxquant}, it is enough to check that for all $q\in[0,1]$, $\tilde \psi(q):=\inf_{\eta\in{\mathcal P}^x}\int_q^1F_\eta^{-1}(p)dp=\int_q^1F_\pi^{-1}(p)dp$ for some probability measure $\pi\in{\mathcal P}_1(\R)$ such that $\int_{\R}y\pi(dy)=x$. For $\eta\in{\mathcal P}^x$, $\int_0^1F_\eta^{-1}(p)dp=x$ and for all $q\in[0,1]$, $\int_q^1F_\eta^{-1}(p)dp\ge (1-q)x$. Therefore for all $q\in[0,1]$, $\tilde \psi(q)\ge (1-q)x$, $\tilde \psi(0)=x$ and $\tilde \psi(1)=0$. The function $\tilde \psi$ being concave on $[0,1]$ as the infimum of concave functions it is continuous on $(0,1)$. Since for $\eta\in{\mathcal P}^x$, $\tilde \psi(q)\le\int_q^1F_\eta^{-1}(p)dp$, $\tilde \psi$ is continuous at $0$ and $1$ and therefore on $[0,1]$. Denoting its left-hand derivative by $f$, one has $\int_0^1|f(p)|dp<\infty$ and for all $q\in[0,1]$, $\tilde \psi(q)=\int_q^1f(p)dp$ with $f$ non-decreasing. One concludes by defining $\pi$ as the image of the Lebesgue measure on $(0,1)$ by $f$.
\end{proof}
\begin{lemma}\label{lem_wass_strict} Let $\rho>1$ and $\eta,\eta_1,\eta_2 \in \cP_\rho(\R^d)$. Then  \begin{equation}\label{ineg_wass}
    W_\rho^\rho\left(\eta,\frac{\eta_1+\eta_2}2\right) \le \frac12\left( W_\rho^\rho(\eta,\eta_1)+W_\rho^\rho(\eta,\eta_2)\right),
  \end{equation}
Besides, when $\eta$ is absolutely continuous with respect to the Lebesgue measure or $d=1$ and $\eta$ has no atom, equality holds if and only if $\eta_1=\eta_2$.  Last, when $d=1$, the statements remain valid with $\frac{\eta_1+\eta_2}{2}$ replaced by the distribution $\bar{\eta}_{12}$ of $\frac{F_{\eta_1}^{-1}+F_{\eta_2}^{-1}}{2}(U)$ with $U$ uniformly distributed on $[0,1]$.
\end{lemma}
\begin{proof}
  Let $\eta_3=\frac{\eta_1+\eta_2}2$. For $i\in\{1,2,3\}$, there exists an optimal probability measure $\pi_i\in\Pi(\eta,\eta_i)$ that satisfies $W_\rho^\rho(\eta,\eta_i)=\int_{\R^d\times\R^d}|y-x|^\rho\pi_i(dx,dy)$. Since $\frac{\pi_1+\pi_2}{2} \in \Pi(\eta,\eta_3)$, we have \begin{align}
    W_\rho^\rho\left(\eta,\frac{\eta_1+\eta_2}2\right)\le \int_{\R^d\times\R^d}|y-x|^\rho\frac{\pi_1+\pi_2}{2}(dx,dy)=  \frac12\left( W_\rho^\rho(\eta,\eta_1)+W_\rho^\rho(\eta,\eta_2)\right).\label{demineqwass}
  \end{align}
We now suppose that $\eta$ is absolutely continuous with respect to the Lebesgue measure.   We know by Theorem~6.2.4 in~\cite{AGS} that the probability measure $\pi_i\in\Pi(\eta,\eta_i)$ satisfying $W_\rho^\rho(\eta,\eta_i)=\int_{\R^d\times\R^d}|y-x|^\rho\pi_i(dx,dy)$ is unique, and writes $\pi_i(dx,dy)=\eta(dx)\delta_{T_i(x)}(dy)$ for some Borel map $T_i:\R^d\rightarrow \R^d$. If ~\eqref{ineg_wass} is an equality, then the inequality in \eqref{demineqwass} is also an equality and, by uniqueness, $\frac{\pi_1+\pi_2}2=\pi_3$. Hence $\eta(dx)\delta_{T_3(x)}(dy)= \eta(dx)\frac12 \left(\delta_{T_1(x)}(dy)+\delta_{T_2(x)}(dy)\right)$, which gives $T_1(x)=T_2(x)=T_3(x)$, $\eta(dx)$-a.e., and implies $\eta_1=\eta_2$. 

When $d=1$, if $\eta$ has no atom, according to Theorem~2.9 in~\cite{santambrogio}, $\pi_i$ is still unique and given by $\eta(dx)\delta_{F_{\eta_i}^{-1}(F_{\eta}(x))}(dy)$, so that the same conclusion holds.
Still when $d=1$, since $F_{\bar{\eta}_{12}}^{-1}=\frac{F_{\eta_1}^{-1}+F_{\eta_2}^{-1}}{2}$, by Proposition 2.17~\cite{santambrogio} and strict convexity of $x\mapsto|x|^\rho$, \begin{align*}
 &W_\rho^\rho(\bar{\eta}_{12},\eta)=\int_0^1\left|\frac{1}{2}(F_{\eta_1}^{-1}(p)+F_{\eta_2}^{-1}(p))-F^{-1}_\eta(p)\right|^\rho dp\\&\le \frac{1}{2}\left(\int_0^1|F_{\eta_1}^{-1}(p)-F^{-1}_\eta(p)|^\rho dp+\int_0^1|F_{\eta_2}^{-1}(p)-F^{-1}_\eta(p)|^\rho dp\right)=\frac{1}{2}\left(W_\rho^\rho(\eta_{1},\eta)+W_\rho^\rho(\eta_{2},\eta)\right)
\end{align*}
with equality iff $dp$ a.e. $F_{\eta_1}^{-1}(p)=F_{\eta_2}^{-1}(p)$ i.e. $\eta_1=\eta_2$.
\end{proof}
\begin{lemma}\label{lemcrois}Let $\rho>1$ and $\mu,\nu\in{\mathcal P}_\rho(\R)$. The function $(0,1)\ni p\mapsto F_{\nu^\rho_{\bar{\mathcal P}(\mu)}}^{-1}(p)-F_\nu^{-1}(p)$ is non-decreasing.
\end{lemma}
\begin{proof}
   It is enough to check that for $\eta\in{\mathcal P}_\rho(\R)\cap\bar{\mathcal P}(\mu)$ such that $p\mapsto F_{\eta}^{-1}(p)-F_\nu^{-1}(p)$ is not non-decreasing then $W_\rho^\rho(\nu,\nu_{\bar{\mathcal P}(\eta)})<W_\rho^\rho(\nu,\eta)$ (indeed $F_{\nu_{\bar{\mathcal P}(\eta)}}^{-1}(p)-F_\nu^{-1}(p)$ is non-decreasing and $\nu_{\bar{\mathcal P}(\eta)}\in\bar{\mathcal P}(\eta)\subset\bar{\mathcal P}(\mu)$). By Proposition 2.17~\cite{santambrogio} and the definition of $\nu_{\bar{\mathcal P}(\eta)}$,
\begin{align*}
   W_\rho^\rho(\nu,\nu_{\bar{\mathcal P}(\eta)})=\int_0^1|F_{\nu_{\bar{\mathcal P}(\eta)}}^{-1}(p)-F_\nu^{-1}(p)|^\rho dp=\int_0^1|f(p)|^\rho dp,
\end{align*}
where $f(p)$ denotes the left-hand derivative of the concave hull $\tilde \psi(q)$ of $[0,1]\ni q\mapsto\phi(q):=\int_q^1F_\eta^{-1}(p)-F_\nu^{-1}(p)dp$. Since $\forall q\in[0,1]$, $\int_q^1F_\eta^{-1}(p)-F_\nu^{-1}(p)dp\le \int_q^1F_\eta^{-1}(p)dp-q\int_0^1
F_\nu^{-1}(p)dp$ where the right-hand side is a concave function of $q$, $\tilde \psi(1)=\phi(1)=0$ and $\tilde \psi(0)=\phi(0)=\int_0^1F_\eta^{-1}(p)-F_\nu^{-1}(p)dp$. Now either $\tilde \psi$ and $\phi$ coincide on $[0,1]$ and $F_\eta^{-1}-F_\nu^{-1}$ is non-decreasing or the open set $\{q\in[0,1]:\tilde \psi(q)>\phi(q)\}$ is non empty and writes as the at most countable union $\bigcup_{i\in I}(p_i,q_i)$ of disjoint intervals with $0\le p_i<q_i\le 1$, $\tilde \psi(p_i)=\phi(p_i)$, $\tilde \psi(q_i)=\phi(q_i)$ and $\tilde \psi$ affine on $[p_i,q_i]$. For each $i$ in the non empty set $I$, for all $p\in (p_i,q_i]$, $f(p)=\frac{\tilde \psi(q_i)-\tilde \psi(p_i)}{q_i-p_i}=\frac{\phi(q_i)-\phi(p_i)}{q_i-p_i}=\frac{\int_{p_i}^{q_i}F_\nu^{-1}(q)-F_\eta^{-1}(q)dq}{q_i-p_i}$ so that, by Jensen's inequality,
\begin{align}
\forall i\in I,\;\int_{p_i}^{q_i}|f(p)|^\rho dp<\int_{p_i}^{q_i}|F_\nu^{-1}(p)-F_{\eta}^{-1}(p)|^\rho dp.\label{strictaff}\end{align}
For $p\in(0,1]\setminus \bigcup_{i\in I}(p_i,q_i]$, either $\tilde \psi$ is equal to $\phi$ on a left-hand neighbourhood of $p$ or there is an accumulation of intervals $((p_{i_n},q_{i_n}))_{n\in\N}$ at the left of $p$ with $(i_n)_{n\in\N}$ a sequence of distinct elements of $I$, $q_{i_n}<p$ for all $n\in\N$ and $\lim_{n\to\infty} q_{i_n}=p$. For $q$ in the left-hand neighbourhood of $p$ in the first case and in $\{q_{i_n}:n\in\N\}$ in the second one, $\tilde{\psi}(p)-\tilde\psi(q)=\phi(p)-\phi(q)=\int_q^p F_\nu^{-1}(r)-F_\eta^{-1}(r)dr$. By the left-continuity of $q\mapsto F_\nu^{-1}(q)-F_\eta^{-1}(q)$ and the definition of $f$, one concludes that $f(p)=F_\nu^{-1}(p)-F_{\eta}^{-1}(p)$.
Therefore $\int_0^11_{\{p\notin\bigcup_{i\in I}(p_i,q_i]\}}|f(p)|^\rho dp=\int_0^11_{\{p\notin\bigcup_{i\in I}(p_i,q_i]\}}|F_\nu^{-1}(p)-F_\eta^{-1}(p)|^\rho dp$ which combined with \eqref{strictaff} and Proposition 2.17~\cite{santambrogio} leads to $\int_0^1|f(p)|^\rho dp<\int_0^1|F_\nu^{-1}(p)-F_{\eta}^{-1}(p)|^\rho dp=W_\rho^\rho(\nu,\eta)$ when $\tilde \psi$ and $\phi$ do not coincide on $[0,1]$. \end{proof}
\begin{remark}Lemma~\ref{lem_couplage_R} can be proved by similar arguments. But to exhibit $\tilde\eta\in\underline{\mathcal P}(\nu)$ with $W_\rho^\rho(\mu,\tilde\eta)\le W_\rho^\rho(\mu,\eta)$ and $F_{\tilde\eta}^{-1}-F_\mu^{-1}$ non-increasing when $\eta\in\underline{\mathcal P}(\nu)$ is such that $F_\eta^{-1}-F_\mu^{-1}$ is not non-increasing, we chose a more elementary transformation exploiting directly the lack of monotonicity in place of $\mu_{\underline{\mathcal P}(\eta)}$.\end{remark}\begin{lemma}\label{prop_irred_comp2}
  Let $\mu,\nu\in \mathcal{P}_1(\R)$ be two distinct probability measures such that $\mu \lecx \nu$ and $(\ut_n, \ot_n)$, $1\le n\le N\in \N^*\cup \{\infty\}$ be the irreducible components of $(\mu,\nu)$. Then, we have 
  $$\left\{ q\in [0,1], \int_0^q F^{-1}_\mu (p)dp > \int_0^q F^{-1}_\nu (p)dp\right\} =  \bigcup_{n=1}^N(F_{\mu}(\ut_n), F_{\mu}(\ot_n-)).$$
\end{lemma}
\begin{proof}For $\eta \in \mathcal{P}_1(\R)$, let $\varphi_\eta(t)=\int_{-\infty}^tF_\eta(x)dx$ for $t\in\R$, $\psi_\eta (q)=\int_0^q F^{-1}_\eta (p)dp$ for $q\in [0,1]$ and $\psi_\eta (q)=+\infty$ for $q \not \in [0,1]$. One has $\varphi_\mu(t)=\int_\R(t-x)^+\mu(dx)\le \int_\R(t-y)^+\nu(dy)=\varphi_\nu(t)$ for all $t\in\R$ and 
$(\ut_n, \ot_n)$, $1\le n\le N\in \N^*\cup \{\infty\}$ is the countable family of disjoint intervals such that \begin{equation}
   \{t\in\R:\varphi_\mu(t)<\varphi_\nu(t)\}=\cup_{n=1}^N(\ut_n, \ot_n)\label{compirr}. 
\end{equation}

 Since $\varphi_\eta$ and $\psi_\eta$ are the antiderivatives of two reciprocal non-decreasing functions, it is well known they are the Legendre-Fenchel transforms of each other i.e. $\varphi_\eta(t)=\sup_{q\in \R}\{qt-\psi_\eta(q)\}$. In fact, for $t\in\R$, if $F_\eta(t-)>0$ then $F_\eta^{-1}(q)<t$ for $q\in (0,F_\eta(t-))$, if $F_\eta(t)<1$ then $F_\eta^{-1}(q)>t$ for $q\in (F_\eta(t),1)$ and if $F_\eta(t-)<F_\eta(t)$ then $F_\eta^{-1}(q)=t$ for $q\in (F_\eta(t-),F_\eta(t)]$. We deduce that $\sup_{q\in \R}\{qt-\psi_\eta(q)\}=F_\eta(t)t-\psi_\eta(F_\eta(t))= \int_0^{F_\eta(t)} (t-F^{-1}_\eta (p))dp=\int_0^1 (t-F^{-1}_\eta (p))^+dp=\varphi_\eta(t)$ and 
\begin{equation}
   \forall t\in\R,\;\{q \in \R, qt-\psi_\eta(q)=\varphi_\eta(t)\}=[F_\eta(t-),F_\eta(t)].\label{ensatteifench}
  \end{equation}
Therefore, we have
    \begin{align*}
      \{t\in \R, \varphi_{\mu}(t)<\varphi_{\nu}(t) \}&\subset \{t \in \R, \forall q \in [F_\nu(t-),F_\nu(t)], qt-\psi_\mu(q)< qt-\psi_\nu(q) \} \\
     &= \{t \in \R, \forall q \in [F_\nu(t-),F_\nu(t)], \psi_\mu(q)> \psi_\nu(q) \}. 
    \end{align*}
Hence
\begin{equation}
   \bigcup_{1\le n\le N}(F_\nu(\ut_n),F_\nu(\ot_n-))\subset \bigcup_{1\le n\le N}\bigcup_{t\in (\ut_n,\ot_n) }   [F_\nu(t-),F_\nu(t)]\subset \{ q\in [0,1], \psi_\mu (q) > \psi_\nu (q)\}.\label{inclus1}
\end{equation} 
    Now, we observe that $(0,1)\subset\cup_{t\in\R}   [F_\mu(t-),F_\mu(t)]$ and, for $t\in \R$ such that $F_\mu(t-)<F_\mu(t)$, $\psi_\mu(q)$ is affine for  $q\in [F_\mu(t-),F_\mu(t)]$. Using the convexity of $\psi_\nu$, we get
 $$\{ q\in [0,1], \psi_\mu (q) > \psi_\nu (q)\}\subset \bigcup_{t \in \R : \psi_\mu (F_\mu(t))>\psi_\nu (F_\mu(t)) \text{ or } \psi_\mu (F_\mu(t-))>\psi_\nu (F_\mu(t-))} [F_\mu(t-),F_\mu(t)].   $$
    If $\psi_\mu (F_\mu(t))>\psi_\nu (F_\mu(t))$, we have $\varphi_\mu(t)=F_\mu(t)t-\psi_\mu(F_\mu(t))<F_\mu(t)t-\psi_\nu(F_\mu(t))\le \varphi_\nu(t)$ by using that $\varphi_\nu$ is the Legendre transform of $\psi_\nu$. Similarly, $\psi_\mu (F_\mu(t-))>\psi_\nu (F_\mu(t-))\implies \varphi_\mu(t)<\varphi_\nu(t)$, and we get
$$\{ q\in [0,1], \psi_\mu (q) > \psi_\nu (q)\} \subset  \bigcup_{t \in \R :   \varphi_\mu(t)<\varphi_\nu(t)} [F_\mu(t-),F_\mu(t)] \subset\bigcup_{1\le n\le N}[F_\mu(\ut_n),F_\mu(\ot_n-)].$$
By \eqref{ensatteifench}, if $\ut_n>-\infty$, $\psi_\mu(F_\mu(\ut_n))=\ut_nF_\mu(\ut_n)-\varphi_\mu(\ut_n)$. Since $\varphi_\mu(\ut_n)=\varphi_\nu(\ut_n)$ and the Legendre transform $\psi_\nu$ of $\varphi_\nu$ is not greater than $\psi_\mu$, we deduce that $\psi_\mu(F_\mu(\ut_n))=\psi_\nu(F_\mu(\ut_n))$. In the same way, if $\ot_n<+\infty$, then $\psi_\mu (F_\mu(\ot_n-))=\psi_\nu (F_\mu(\ot_n-))$ so that 
\begin{equation}
   \{ q\in [0,1], \psi_\mu (q) > \psi_\nu (q)\} \subset\bigcup_{1\le n\le N}(F_\mu(\ut_n),F_\mu(\ot_n-)).\label{inclus2}
\end{equation}
Now, \eqref{compirr} implies that $F_\mu(\ut_n)\le F_\nu(\ut_n)$. If $F_\mu(\ut_n)< F_\nu(\ut_n)$, we necessarily have $\ut_n>-\infty$, and for $q\in (F_\mu(\ut_n), F_\nu(\ut_n))$, we have $F_\nu^{-1}(q)\le \ut_n $ and $F_\mu^{-1}(q)> \ut_n $ since $F_\mu$ is right-continuous. Therefore, we have $\int_{F_\mu(\ut_n)}^p F_\mu^{-1}(q) dq >\int_{F_\mu(\ut_n)}^p F_\nu^{-1}(q) dq$ and thus $\psi_\mu(p)>\psi_\nu(p)$ for $p\in(F_\mu(\ut_n),F_\nu(\ut_n)]$. Similarly, we show that  $\psi_\mu(q)>\psi_\nu(q)$ for $q\in[F_\nu(\ot_n-),F_\mu(\ot_n-))$, which, with \eqref{inclus1} and \eqref{inclus2}, gives the claim.
\end{proof}

\bibliographystyle{plain}
\bibliography{Biblio}

\end{document}